\documentclass[11 pt]{article}
\usepackage{style}
\title{\LARGE\bfseries Query-Efficient Zeroth-Order Algorithms for Nonconvex Constrained Optimization}
\author{
Ruiyang Jin\textsuperscript{$\dagger,1$}, Yuke Zhou\textsuperscript{$\ddagger,1$}, Yujie Tang\textsuperscript{$\ddagger,2$}, Jie Song\textsuperscript{$\ddagger,3$}, and Siyang Gao\textsuperscript{$\dagger,2$}\\
{\small
\textsuperscript{$\dagger$}\textit{{City University of Hong Kong,}}
\href{mailto:ruiyajin@cityu.edu.hk}{\textit{\textsuperscript{$1$}ruiyajin@cityu.edu.hk}}, 
\href{mailto:siyangao@cityu.edu.hk}{\textit{\textsuperscript{$2$}siyangao@cityu.edu.hk}}}\\
{\small
\textsuperscript{$\ddagger$}\textit{Peking University,} \href{mailto:zhouyk@pku.edu.cn}{\textit{\textsuperscript{$1$}zhouyk@pku.edu.cn}}, 
\href{mailto:yujietang@pku.edu.cn}{\textit{\textsuperscript{$2$}yujietang@pku.edu.cn}},
\href{mailto:jie.song@pku.edu.cn}{\textit{\textsuperscript{$3$}jie.song@pku.edu.cn}}}
}
\date{\vspace{-0.4 in}}

\begin{document}
\maketitle

\begin{abstract}
Zeroth-order optimization (ZO) has been a powerful framework for solving black-box problems, which estimates gradients using zeroth-order data to update variables iteratively. The practical applicability of ZO critically depends on the efficiency of single-step gradient estimation and the overall query complexities. However, existing constrained ZO algorithms cannot achieve efficiency on both simultaneously. In this work, we consider a general constrained optimization model with black-box objective and constraint functions. To solve it, we propose novel algorithms that can achieve the best-known overall query complexity bound of $\mathcal{O}(d/\epsilon^4)$ to find an $\epsilon$-stationary solution ($d$ is the dimension of variables), while reducing the queries for estimating a single-step gradient from $\mathcal{O}(d)$ to $\mathcal{O}(1)$. Specifically, we integrate block gradient estimators with gradient descent ascent, which leads to two algorithms, ZOB-GDA and ZOB-SGDA, respectively. Instead of constructing full gradients, they estimate only partial gradients along random blocks of dimensions, where the adjustable block sizes enable high single-step efficiency without sacrificing convergence guarantees. Our theoretical results establish the finite-sample convergence of the proposed algorithms for nonconvex optimization. Finally, numerical experiments demonstrate the superior performance of our algorithms compared to existing methods.
\end{abstract}

\section{Introduction}\label{sec:intro}
In practical problems, it is common to encounter real systems that lack analytical expressions or models. In such cases, only zeroth-order (input-output) information of the systems is accessible. The lack of higher-order information makes it especially difficult to optimize these systems. In this research, we consider a general constrained optimization model for these problems:
\begin{equation}
\begin{aligned}\label{eq:constrained_problem}
    \min_{x\in \mathbb{R}^{d_x}} h(x)\qquad \text{s.t.}\; c_j(x)\leq 0,\; \forall j\in\mathcal{J},
\end{aligned}    
\end{equation}
where $h:\mathbb{R}^{d_x}\to\mathbb{R}$ is the objective function and each $c_j:\mathbb{R}^{d_x}\to \mathbb{R},\forall j\in\mathcal{J}$ is a constraint function. In our problem, we have the variable dimension $d=d_x$. Both $h(x)$ and $c_j(x)$ do not have analytical expressions and are treated as black boxes, i.e., only the input $x$ and the corresponding deterministic function outputs $h(x)$ or $c_j(x)$ are observable. Neither $h$ nor $c_j,\forall j\in\mathcal{J}$ is necessarily convex.

Problems of the form (\ref{eq:constrained_problem}) arise in many domains, such as power systems \citep{hu2023gradient,zhou2025zeroth}, simulation optimization \citep{park2015penalty}, and machine learning \citep{nguyen2023stochastic}. However, traditional model-based or gradient-based algorithms are generally not applicable to problem (\ref{eq:constrained_problem}), as they rely on first-order or second-order information (e.g., gradients or Hessians) of $h(x)$ and $c_j(x)$ that is not available.
Zeroth-order optimization (ZO), a representative method in derivative-free optimization, offers a promising approach to this type of optimization problem, and has been broadly applied \citep{fu2015handbook,liu2020primer,malladi2023fine,lam2025distributionally}. The fundamental idea behind ZO is to construct estimators of first-order information using zeroth-order data \citep{berahas2022theoretical}, and integrate these estimators into gradient-based algorithms, such as gradient descent, to seek optimal or high-quality solutions.

Under the iterative ZO framework, the efficiency of single-step gradient estimation and overall query complexity jointly determine the practical applicability of ZO. These two metrics refer to the number of function values required to generate a gradient estimator and a final solution, respectively. The traditional \textit{coordinate-wise gradient estimation} (CGE) requires estimating partial gradients along all dimensions separately based on finite differences of function values \citep{kiefer1952stochastic}. CGE-based algorithms can generally enjoy state-of-the-art overall query complexity bounds due to the controllable bias and variance of CGEs \citep{xu2024derivative}. However, each CGE costs $\mathcal{O}(d)$ function evaluations ($d$ is the dimension of the variable space), which becomes inefficient in high-dimensional problems, even if parallel computing techniques are available \citep{scheinberg2022finite,hu2025convergence}. In contrast, the prevalent \textit{randomized gradient estimation} (RGE) only requires one or two function values to construct a gradient estimator along a random direction \citep{flaxman2005online,nesterov2017random}. RGE-based algorithms have demonstrated excellent performance in unconstrained problems. However, they suffer from slow convergence when applied to constrained cases (such as (\ref{eq:constrained_problem})) due to the large variance of gradient estimations \citep{liu2020primer}. Existing methods must trade off low single-step efficiency against good overall complexity, and none of these methods is query-efficient in both for high-dimensional constrained problems.

In view of this challenge, a fundamental question arises:
\emph{To solve (\ref{eq:constrained_problem}),  is it possible to achieve the best-known query complexity bound while improving the single-step efficiency?} Similar challenges have also been recognized in \cite{scheinberg2022finite} and \cite{hu2025convergence}. In this study, we provide a positive answer to this question by utilizing the framework of random block updates to design query-efficient ZO algorithms for solving (\ref{eq:constrained_problem}). We show that our method can enjoy improved single-step efficiency and the best-known overall query complexity bound.

\subsection{Main Contributions}
We assume an oracle that returns $h(x)$ and $c_j(x),\forall j\in \mathcal{J}$ (i.e., we can observe all the function evaluations of $h(x)$ and $c_j(x),\forall j\in \mathcal{J}$ simultaneously via querying a $x$) but no gradient information. To handle the black-box constraints in (\ref{eq:constrained_problem}), we adopt a primal-dual framework by reformulating it as a deterministic min-max problem:
\begin{align}\label{eq:general_minmax}
 \min_{x\in \mathbb{R}^{d_x}}\max_{y\in \mathcal{Y}} f(x,y),
\end{align}
where $f(x,y)=h(x)+y^\text{T} c(x)$ is the Lagrange function of problem (\ref{eq:constrained_problem}). Here, $c(x)=(c_1(x),\cdots,c_{d_y}(x))^\text{T}$ with $d_y=|\mathcal{J}|$; $\mathcal{Y}$ is the feasible set of Lagrange multipliers. Clearly, $f(x,y)$ is \textit{nonconvex-concave}, i.e., nonconvex in $x$ and concave in $y$, when $h(x)$ and $c_j(x),\forall j\in\mathcal{J}$ are not assumed convex. Then, solving problem (\ref{eq:general_minmax}) can provide high-quality solutions to problem (\ref{eq:constrained_problem}) \citep{nesterov2018lectures}. The detailed contributions of this work are summarized as follows.

\begin{table}[htbp]
\caption{Comparison of single-step and overall query complexity bounds}
\label{tab:comparison}
\centering
\renewcommand{\arraystretch}{1.1}
\begin{tabular}{p{4cm}ccc}
\hline
Algorithms & Gradient Estimator & Queries per Step & Overall Queries \\ \hline
SZO-ConEx \citep{nguyen2023stochastic} & RGE & $\mathcal{O}(1)$ & $\mathcal{O}(d/\epsilon^6)$ \\
ZOAGP \citep{xu2024derivative} & CGE & $\mathcal{O}(d)$ & $\mathcal{O}(d/\epsilon^4)$ \\
ZOB-GDA (Ours) & BCGE & $\mathcal{O}(b)$ & $\mathcal{O}(d/\epsilon^6)$ \\
ZOB-SGDA (Ours) & BCGE & $\mathcal{O}(b)$ & $\mathcal{O}(d/\epsilon^4)$ \\
\hline
\end{tabular}

\vspace{0.5ex}
\footnotesize
\noindent
Note: In our algorithms, $d=d_x$. $b$ is the block size that can be chosen from $\{1,2,\ldots,d\}$. Overall queries refer to the theoretical upper bound on the number of oracle queries required to achieve an $\epsilon$-stationary point.
\end{table}

\paragraph{Leveraging Block Updates with Zeroth-Order Algorithms to Solve (\ref{eq:constrained_problem}).} In this research, we adopt the widely-used gradient descent ascent (GDA) framework to solve problem (\ref{eq:constrained_problem}). However, directly applying RGE or CGE in GDA cannot exhibit satisfactory performance in both single-step and overall complexity \citep{hu2025convergence}. To address this, we combine the framework of random block updates with GDA and smoothed GDA to develop two novel algorithms, called \textit{zeroth-order block gradient descent ascent} (ZOB-GDA) and \textit{zeroth-order block smoothed gradient descent ascent} (ZOB-SGDA). Rather than estimating a full gradient at each step, they randomly select a block of coordinates and update the variables using \textit{block coordinate-wise gradient estimations} (BCGEs). The adoption of the BCGEs effectively controls the bias and variance of gradient estimations to be negligible and thereby accelerates convergence. Moreover, the block size is adjustable to control the number of queries required to construct a single-step gradient.

\paragraph{Best-Known Query Complexity Bounds with Controllable Single-Step Efficiency.} We establish finite-sample guarantees for the proposed algorithms by analyzing the min-max problems (\ref{eq:general_minmax}) in nonconvex-concave settings. The query complexity bounds are summarized in Table \ref{tab:comparison} and compared with two representative algorithms. Specifically, ZOB-GDA can find an $\epsilon$-stationary point of $f(x,y)$ with a query complexity bound $\mathcal{O}(d/\epsilon^6)$, which differs from the bound for first-order GDA only by a factor of $d$. Moreover, ZOB-SGDA is shown to have the query complexity bound $\mathcal{O}(d/\epsilon^4)$, which is consistent with the best available results in the literature for solving deterministic nonconvex-concave problems. Different from \cite{xu2024derivative}, our algorithms benefit from controllable efficiency in single-step gradient estimation, which makes them query-efficient for both single-step and overall complexities. The numerical results demonstrate that our algorithms require significantly fewer queries for both a single step and overall convergence compared to existing methods.

\subsection{Related Work}
Here, we review the related work on ZO, zeroth-order GDA, and zeroth-order coordinate/block updates.
\paragraph{Zeroth-Order Optimization.} ZO has emerged as a prevalent tool to solve black-box problems and found application across machine learning \citep{liu2020primer, nguyen2023stochastic}, power systems \citep{hu2023gradient,jin2026zeroth}, simulation optimization \citep{xu2023gradient,lam2025distributionally}, large language models \citep{malladi2023fine,pmlr-v235-zhang24ad}, etc. ZO originates from the stochastic approximation method in \cite{kiefer1952stochastic}, where CGE is applied to estimate partial gradients along all dimensions via finite differences of function values. This is inefficient for high-dimensional problems, even if parallel techniques can be applied \citep{scheinberg2022finite}.  To improve single-step efficiency, one-point and two-point RGEs have been developed by estimating gradients along randomized directions \citep{flaxman2005online,nesterov2017random, lam2025distributionally}. Generally, RGE-based algorithms can achieve the same oracle complexities as their first-order counterparts in unconstrained problems, differing by a dimension-dependent factor \citep{liu2020primer}. However, RGEs suffer from large variance in gradient estimation in constrained problems (see Section \ref{subsec:zoge} for a detailed discussion). Moreover, most literature considers simple constraints on the input $x$ that can be dealt with by projection operations \citep{duchi2015optimal,yuan2015zeroth,jin2026zeroth,he2024parallel}, which cannot be used to solve problem (\ref{eq:constrained_problem}).

\paragraph{Zeroth-Order GDA.} GDA is a classical framework for solving min-max problems and has been extensively studied \citep{nemirovski2004prox, nedic2009subgradient,lin2020gradient,xu2023unified,zhang2020single}.
It is also well-established and widely applied to solve zeroth-order min-max problems of the form (\ref{eq:general_minmax}) \citep{hu2023gradient,zhou2025zeroth}. The authors of \cite{liu2018zeroth} applied the two-point RGE to solve a composite optimization problem. Then, the standard zeroth-order GDA was applied to the general min-max problems \citep{liu2020min,wang2023zeroth}, while only the query complexity bound of nonconvex-strongly concave cases was established. Several variants of zeroth-order GDA have been developed for convex-concave settings, such as zeroth-order OGDA-RR \citep{maheshwari2022zeroth} and zeroth-order extra-gradient \citep{zhou2025zeroth}, which can achieve the query complexity bounds of $\mathcal{O}(d^4/\epsilon^2)$ and $\mathcal{O}(d/\epsilon^2)$ to derive an $\epsilon$-optimal solution, respectively. For nonconvex-concave problems, \cite{nguyen2023stochastic} and \cite{an2024robust} applied RGEs to design algorithms that achieve the query complexity bound of $\mathcal{O}(d/\epsilon^6)$ to derive an $\epsilon$-stationary/KKT point. In addition, \cite{xu2024derivative} proposed combining alternating gradient projection with CGEs to solve a min-max problem with a complexity bound of $\mathcal{O}(d/\epsilon^4)$ (matching the best available bound) to obtain an $\epsilon$-stationary point.

\paragraph{Zeroth-Order Coordinate/Block Updates.} 
The framework of coordinate/block updates is widely adopted in first-order optimization \citep{nesterov2017efficiency, latafat2019new}. The core idea is to apply the partial gradients along a subset of full dimensions to update the iterates. The applications of coordinate/block updates in ZO mainly lie in unconstrained problems \citep{lian2016comprehensive,cai2021zeroth}, where only the coordinate/block gradients along a subset are estimated using coordinate/block CGEs or RGEs at each step. Their extension to constrained problems, however, remains relatively underexplored. In \cite{shanbhag2021zeroth}, the RGE was combined with zeroth-order block updates and projected gradient descent to solve a stochastic constrained problem. Moreover, in \cite{he2024parallel} and \cite{jin2026zeroth}, a cyclic zeroth-order block coordinate descent method and a randomized zeroth-order coordinate descent method were proposed, respectively, to solve the deterministic constrained problems and achieve complexity bounds proportional to \( \epsilon^{-2} \) for nonconvex optimization. However, all these methods require the constraint set to be projection-friendly and coordinate/block-structured, which is usually too restrictive in practical problems and inapplicable to non-analytical constraint sets (such as in our problem (\ref{eq:constrained_problem})).

\section{Preliminaries}
For a positive integer $n$, we denote $[n]:=\{1,2,\cdots,n\}$. For a vector $x\in\mathbb{R}^{d_x}$, denote $(x)_i$ as its $i$th entry. For a differentiable function $h(x):\mathbb{R}^{d_x}\to \mathbb{R}$, denote $\nabla h(x)$ as its gradient at $x$ and $\nabla_i h(x),i\in [d_x]$ as the partial gradient along the $i$th dimension. Similarly, for a differentiable function $f(x,y):\mathbb{R}^{d_x}\times \mathbb{R}^{d_y}\to\mathbb{R}$, denote the partial gradient w.r.t. $x$ (and $y$) by $\nabla_{x}f(x,y)$ (and $\nabla_{y}f(x,y)$). Without further specification, $\|\cdot\|$ denotes the $\ell_2$-norm in Euclidean space. The Euclidean projection operator onto a closed convex set $\mathcal{X}$ is denoted by $\mathcal{P}_\mathcal{X}[\cdot]$. 

\subsection{Assumptions and Stationarity Measure}
Below, we present the key assumptions for our analysis and introduce the definition of stationarity measure for evaluating our proposed algorithms.

\begin{assumption}\label{assump:boundedness}
    The set $\mathcal{Y}$ is compact, i.e., $\mathcal{Y}:=\{y\in \mathbb{R}^{d_y}| 0 \leq y\leq\overline{y}\}$ for some bounded $\overline{y}\in\mathbb{R}^{d_y}$. Moreover, $\Phi (x)=\max_{y\in\mathcal{Y}}f(x,y)$ is lower bounded by some finite constant $\underline{f}$.
\end{assumption}
\begin{remark}
A sufficient condition for the lower boundedness of $\Phi(x)$ is that $h(x)$ is lower bounded for any $x\in\mathbb{R}^{d_x}$. In problem (\ref{eq:general_minmax}), the natural domain of Lagrange multipliers is $\mathbb{R}^{d_y}_+$. Nevertheless, restricting the dual iterates to a compact set $\mathcal{Y}$ that contains at least one optimal multiplier is standard in primal–dual analysis. Under Slater condition, the boundedness of the optimal dual set has been justified in \cite{nedic2009subgradient}. This assumption is commonly imposed in the analysis of GDA-type algorithms and is essential for establishing finite-time convergence guarantees \citep{liu2020min,xu2023unified}.

\end{remark}

\begin{assumption}\label{assump:lipschitz}
    $f(x,y)$ is differentiable and Lipschitz continuous, i.e., $\forall (x_1,y_1), (x_2,y_2)\in\mathbb{R}^{d_x}\times\mathcal{Y}$, we have $| f(x_1,y_1)-f(x_2,y_2)| \leq \Lambda\|(x_1,y_1)-(x_2,y_2)\|$ for some $\Lambda>0$.
\end{assumption}
\begin{assumption}\label{assump:smoothness}
$f(x,y)$ is $L$-smooth in $x$ and $y$, i.e., there exist some $L> 0$ satisfying $\| \nabla f(x_1,y_1)-\nabla f(x_2,y_2) \|\leq L\|(x_1,y_1)-(x_2,y_2)\|$ for any $x_1,x_2\in\mathbb{R}^{d_x}$, and $y_1,y_2\in\mathcal{Y}$.
\end{assumption}
Assumptions \ref{assump:lipschitz}-\ref{assump:smoothness} impose the Lipschitz continuity on $f(x,y)$ and its gradients, which are standard in the literature of both first-order and zeroth-order optimization \citep{nedic2009subgradient,ghadimi2013stochastic,zhou2025zeroth}. Similarly, we can also equivalently impose Lipschitz continuity on $h(x),\;c_j(x),\forall j\in \mathcal{J}$ and their gradients to replace Assumptions \ref{assump:lipschitz} and \ref{assump:smoothness}.

For min-max problems, a widely adopted stationarity measure is the proximal gradient for first/zeroth-order nonconvex optimization \citep{lin2020gradient,liu2020min, xu2023unified}:
$$
\mathfrak{g}(x,y)=
\begin{pmatrix}
    \mathfrak{g}_x(x,y) \\
    \mathfrak{g}_y(x,y)
\end{pmatrix}=
\begin{pmatrix}
    \nabla_x f(x,y) \\
    (1/\beta)\left(y-\mathcal{P}_\mathcal{Y} \left[y+\beta \nabla_y f(x,y)\right]\right)
\end{pmatrix},
$$
where $\beta$ is the step size for dual updates.
A point $(x,y)\in \mathbb{R}^{d_x}\times \mathcal{Y}$ is a first-order stationary point of (\ref{eq:general_minmax}) if $\|\mathfrak{g}(x,y)\|=0$.

We also introduce another notion of stationarity measure. The problem (\ref{eq:general_minmax}) is equivalent to minimizing the function $\Phi(x) = \max_{y \in \mathcal{Y}} f(x, y)$ over $\mathbb{R}^{d_x}$. The norm of $\nabla \Phi(x)$ is an appropriate stationarity measure for nonconvex optimization when $\Phi(x)$ is differentiable \citep{wang2023zeroth}. However, $\Phi(x)$ may fail to be differentiable even if $f(x, y)$ is concave in $y$, as the maximum may not be uniquely attained. Alternatively, we define the Moreau envelope of $\Phi(x)$ for any $\lambda>0$ as 
\[
\Phi_{\lambda}(x) = \min_{u \in \mathbb{R}^{d_x}} \big\{\Phi(u) + \frac{1}{2\lambda} \|u - x\|^2 \big\}.
\]
Under Assumption \ref{assump:smoothness}, the Moreau envelope $\Phi_{1/2L}(x)$ with parameter $\lambda=\frac{1}{2L}$ and $\nabla \Phi_{1/2L}(x)$ are both well-defined because $\Phi (u)+L\|u-x\|^2$ is strongly convex in $u$ given $x$. Furthermore, $\Phi_{1/2L}(x)$ is differentiable and smooth in $x$. A point $x\in \mathbb{R}^{d_x}$ is a stationary point of $\Phi$ if $\|\nabla \Phi_{1/2L}(x)\| =0$. This stationarity measure is also widely used in nonconvex-concave settings \citep{mahdavinia2022tight, davis2019stochastic}. As shown in \cite{lin2020gradient}, the computational overhead of transferring this notion to the one measured by $\|\mathfrak{g}(x,y)\|$ is negligible relative to the overall query complexity. 
With this minor distinction, we define our stationarity notion as follows.

\begin{definition}
We say a point $(x,y)\in\mathbb{R}^{d_x}\times \mathcal{Y}$ is an $\epsilon$-stationary point of (\ref{eq:general_minmax}) if $\|\mathfrak{g}(x,y)\|\leq \epsilon$ or $\|\nabla \Phi_{1/2L}(x)\|\leq \epsilon$.
\end{definition}
In this study, we therefore adopt different metrics of stationarity depending on the analysis techniques, since both metrics are commonly applied.

\subsection{Zeroth-Order Gradient Estimation}\label{subsec:zoge}
Various gradient estimators for ZO have been proposed in the literature, where the two-point RGE is most widely applied. For a differentiable function $h:\mathbb{R}^{d_x}\to \mathbb{R}$, the two-point RGE is defined as
\begin{align}\label{eq:2_point_estimator}
    g(x;r,z)=\frac{h(x+rz)-h(x)}{r}\cdot z,
\end{align}
where $g(x;r,z)\in\mathbb{R}^{d_x}$ is the estimated gradient for $\nabla h(x)$. Here, $r>0$ is the smoothing radius and $z\in \mathbb{R}^{d_x}$ is a random perturbation vector typically sampled from the Gaussian distribution $\mathcal{N}(0,I_{d_x})$ or the uniform distribution on a sphere with radius $\sqrt{d_x}$ \citep{nesterov2017random,duchi2015optimal}. 
In contrast, CGE adds perturbation to each dimension separately and applies the finite difference of function values to construct a full gradient. The CGE is defined as
\begin{align}\label{eq:2_point_coordinate_estimator}
    g(x;r,\{e_i\}_{i=1}^{d_x})=\sum_{i\in [d_x]}\frac{h(x+re_i)-h(x)}{r}\cdot e_i,
\end{align}
where $g(x;r,\{e_i\}_{i=1}^{d_x})\in\mathbb{R}^{d_x}$ approximates $\nabla h(x)$, and $e_i\in \mathbb{R}^{d_x}$ is the unit vector with only the $i$th entry being 1. Let $g_i(x;r,e_i)$ denote the $i$th entry of $g(x;r,\{e_i\}_{i=1}^{d_x})$.

\paragraph{Challenge of Trading off Single-Step and Overall Query Complexities.} It has been shown that the biases of both \eqref{eq:2_point_estimator} and \eqref{eq:2_point_coordinate_estimator} are bounded by $\mathcal{O}(\sqrt{d_x}r)$ \citep{berahas2022theoretical}, which means the biases are negligible when a small $r$ is applied. RGE in (\ref{eq:2_point_estimator}) is efficient for a single step and only requires two function values to construct a gradient. Its variance approximately takes the form $\mathcal{O}(d_x\|\nabla h(x)\|^2+ d_x^3r^2)$ \citep{liu2020primer}. In unconstrained problems, we have $\|\nabla h(x^*)\|=0$ for any optimal solution $x^*$. Therefore, the variance of RGE shrinks as the iterations approach the optimal solution, which allows RGE-based algorithms to mimic their first-order counterparts and achieve similar convergence results. However, in constrained problems, the above property does not hold, as the gradient $\nabla h(x^*)$ may not be zero. The large variance of RGE leads to worse overall query complexity bounds in constrained problems \citep{jin2026zeroth}. In contrast, the variance of CGE is controlled by the order of $\mathcal{O}(d_xr^2)$ and is negligible with small $r$. Therefore, CGE-based algorithms generally enjoy the state-of-the-art overall query complexity bounds \citep{xu2024derivative}. However, CGE requires $\mathcal{O}(d_x)$ function values to construct a full gradient, which is inefficient for a large $d_x$.

As a result, achieving efficiency in both aspects has not been well addressed. Below, we will show that BCGE is a promising answer to this open question when applied to solve \eqref{eq:constrained_problem}.

\section{Zeroth-Order Block Gradient Descent Ascent}
In this section, we leverage BCGEs and block updates to design a new algorithm with controllable single-step efficiency. Then, we establish its convergence guarantee and query complexity bound.

\subsection{Algorithm Design}
We propose the \textit{zeroth-order block gradient descent ascent} (ZOB-GDA) algorithm as presented in Algorithm \ref{algorithm:ZOB-GDA} for solving (\ref{eq:general_minmax}). Our algorithm follows the main steps of standard GDA, which perform the gradient descent in $x$ and gradient ascent in $y$. However, ZOB-GDA differs from the conventional zeroth-order GDA method introduced in \cite{liu2020min}, which employs the RGE in (\ref{eq:2_point_estimator}). Instead of estimating a full gradient, ZOB-GDA randomly selects a block of dimensions to estimate the partial gradients and performs block descent ascent at each step. Specifically, for the $k$th iterate $(x_k,y_k)$, we uniformly sample a block of dimensions, denoted by a set $\mathcal{I}_k\subseteq [d_x]$ with $|\mathcal{I}_k|=b$ ($b$ is the block size ranging from 1 to $d_x$). Then, we estimate the BCGE of $f(x_k,y_k)$ by:
$$G_x^{\mathcal{I}_k}(x_k,y_k)=\sum_{i\in \mathcal{I}_k}\frac{f(x_k+r_ke_i,y_k)-f(x_k,y_k)}{r_k}e_i,$$
where we denote $G_{x}^{\mathcal{I}_k}(x_k,y_k)\in \mathbb{R}^{d_x}$ as the vector with the entries of $\mathcal{I}_k$ being estimated and other entries being 0. We can apply different smoothing radii $r_k$ for each iteration. For simplicity, we denote $G_x(x_k,y_k)=G_x^{\mathcal{I}}(x_k,y_k)$ when $\mathcal{I}=[d_x]$. For the update of the dual variable $y$, we have the partial gradient
$\nabla_y f(x_k,y_k)=c(x_k)$. Therefore, no additional queries are required as $c(x_k)$ has been observed when computing $G_x^{\mathcal{I}_k}(x_k,y_k)$. Therefore, $b+1$ queries are required for each iteration.

Specifically, when the block size satisfies $b=1$, the primal update resembles a coordinate update. When $b=d_x$, the primal update uses a full gradient and resembles the primal update in traditional CGE-based algorithms \citep{xu2024derivative,zhou2025zeroth}. Moreover, we can control the single-step efficiency by adjusting the block size $b$. It will be shown that the choice of $b$ does not affect the theoretical query complexity bound of ZOB-GDA.

Unlike existing ZO literature that applies coordinate/block updates, we first generalize them in the GDA framework to address non-analytical constraints, whose dynamics and convergence analysis are significantly more complicated due to the coupling of primal and dual steps.

\begin{algorithm}[htbp]
	\caption{Zeroth-order block gradient descent ascent (ZOB-GDA)}
	\label{algorithm:ZOB-GDA} 
	\begin{algorithmic}[1]
		\STATE \textbf{Input:} Initial $(x_0,y_0)\in\mathbb{R}^{d_x}\times\mathcal{Y}$, maximum steps $K$, block size $b$, and the step sizes $\alpha, \beta$.
        \FOR{$k=0,1,2,\cdots,K-1$}
		\STATE Uniformly sample $\mathcal{I}_k\subseteq[d_x]$ with $|\mathcal{I}_k|=b$ and update $x_k$ by
        \begin{align}\label{eq:primal_update}
        x_{k+1}=x_k-\alpha\cdot G_x^{\mathcal{I}_k}(x_k,y_k).
        \end{align}

        \STATE Update $y_k$ by $y_{k+1}=\mathcal{P}_{\mathcal{Y}}\left[y_k+\beta\cdot \nabla_y f (x_k,y_k)\right].$
		\ENDFOR
        \STATE \textbf{Output:} $\{(x_k,y_k)\}_{k=0}^K$
	\end{algorithmic}
\end{algorithm}

\subsection{Convergence Results of ZOB-GDA}
In this subsection, we establish convergence guarantees and query complexity bounds for Algorithm~\ref{algorithm:ZOB-GDA} in nonconvex-concave cases. First, we provide the following lemma to bound the bias of coordinate gradient estimation.
\begin{lemma}\label{lemma:estimation_error}
For a $L$-smooth and differentiable function $h:\mathbb{R}^{d_x}\to\mathbb{R}$, i.e., $\|\nabla h(x)-\nabla h(x')\|\leq L\|x-x'\|, \forall x,x'\in\mathbb{R}^{d_x}$, we have $\left| \nabla_i h(x)-g_i(x;r,e_i) \right|\leq \frac{1}{2}Lr$.
\end{lemma}
Lemma \ref{lemma:estimation_error} and its extended versions have appeared in existing literature \citep{lian2016comprehensive, berahas2022theoretical,jin2026zeroth}, thus, we omit its proof here. Lemma \ref{lemma:estimation_error} also implies that $\left| \nabla_i h(x)-g_i(x;r,e_i) \right|^2\leq \frac{1}{4}L^2r^2$. It demonstrates that $g_i(x;r,e_i)$ is a good partial gradient estimator, in the sense that both its bias and variance can be effectively controlled by the smoothing radius. This error bound plays a fundamental role and will be frequently used in our theoretical analysis.

Let $N=\frac{d_x}{b}$, $D_y=\|\overline{y}\|$ and $\Lambda_0=\Lambda+\frac{\sqrt{b}L}{2}\cdot\sup_k \{r_k\}$. Then, we present the following Theorem to characterize the convergence of ZOB-GDA in the sense of $\left\| \nabla \Phi_{1/2L}(x_k)\right\|$.

\begin{theorem}\label{thm:ZOB-GDA}
Suppose Assumptions \ref{assump:boundedness}-\ref{assump:smoothness} hold. The sequence $\{(x_k,y_k)\}_{k=0}^{K}$ is generated by ZOB-GDA. The step sizes satisfy $0<\alpha, \beta\leq 1/L$, and the sequence of smoothing radii satisfies $\sum_{k=0}^K r_k^2\leq \frac{1}{b}$. Then, we have $$\min_{k\leq K-1}\mathbb{E}\left[\left\| \nabla \Phi_{1/2L}(x_k)\right\|\right]\leq \mathcal{O}\left(\sqrt{\frac{N}{\alpha K}}\right)+\epsilon_c,$$
where $\epsilon_c=\big(16 L\Lambda_0 D_y\sqrt{2\alpha/\beta}+32\alpha L\Lambda_0^2\big)^{1/2}$.
\end{theorem}

The results in Theorem \ref{thm:ZOB-GDA} imply that ZOB-GDA can converge to an $\epsilon_c$-stationary point at a convergence rate of $\mathcal{O}(\sqrt{N/\alpha K})$. The fixed error $\epsilon_c$ results from the use of constant step sizes for both primal and dual steps, which is also observed for the first-order GDA \citep{nedic2009subgradient}. We provide the proof of Theorem \ref{thm:ZOB-GDA} in Appendix \ref{app:proof_of_thm_ZOBGDA}.

Given the expression of $\epsilon_c$, one could adopt diminishing or small step sizes for $\alpha$ to eliminate the fixed term $\epsilon_c$. The following corollary characterizes the convergence guarantee of ZOB-GDA.

\begin{corollary}\label{coro:ZOB-GDA_convergence}
Suppose that the conditions in Theorem \ref{thm:ZOB-GDA} hold. Further set the step size $\alpha= \mathcal{O}\big((N/K)^{\frac{2}{3}}\big)$. Then, we have $\min_{k\leq K-1}\mathbb{E}\left[\left\| \nabla \Phi_{1/2L}(x_k)\right\|\right]\leq \mathcal{O}\big((N/K)^{\frac{1}{6}}\big)$.
\end{corollary}
The derivation of Corollary \ref{coro:ZOB-GDA_convergence} is straightforward by substituting the step size into the result in Theorem \ref{thm:ZOB-GDA}. Since the step size $\alpha$ is dependent on the maximum steps $K$, we can apply diminishing settings for $\alpha$ in practice.
Corollary \ref{coro:ZOB-GDA_convergence} shows that using two-time-scale step sizes for the updates of $x$ and $y$ can effectively eliminate the fixed error term $\epsilon_c$. 
Building upon Corollary \ref{coro:ZOB-GDA_convergence}, we obtain the following corollary to establish the overall query complexity bound of ZOB-GDA.

\begin{corollary}\label{coro:ZOB-GDA_complexity}
Suppose the conditions in Theorem \ref{thm:ZOB-GDA} hold. Set $\alpha= \mathcal{O}\left(\epsilon^4\right)$ for any sufficiently small $\epsilon>0$. Then, the query complexity to achieve $\min_{k\leq K-1}\mathbb{E}\left[\left\| \nabla \Phi_{1/2L}(x_k)\right\|\right]\leq\epsilon$ is $\mathcal{O}\left(\frac{d_x}{\epsilon^6}\right)$.
\end{corollary}

To the best of our knowledge, Corollary \ref{coro:ZOB-GDA_complexity} establishes the first query complexity bound for zeroth-order algorithms in the standard GDA framework for nonconvex-concave settings. Notably, this complexity bound differs from the first-order GDA by an additional factor $d_x$ \citep{lin2020gradient}, which is inherent to zeroth-order gradient estimation and occurs commonly in zeroth-order optimization \citep{liu2020primer}. Compared to the ZOAGP algorithm \citep{xu2024derivative} with the query complexity bound $\mathcal{O}(d/\epsilon^4)$, ZOB-GDA's complexity bound seems worse due to the limitation of standard GDA frameworks. However, the single-step gradient estimation can be significantly more efficient for ZOB-GDA by using a small block size. In the next section, we will leverage block updates with a variant of the GDA framework to design a new algorithm that achieves both the best-known overall query complexity bound and adjustable single-step efficiency.

\section{Zeroth-Order Block Smoothed Gradient Descent Ascent}
In this section, we leverage block updates with a variant of GDA, smoothed GDA, to design a new algorithm, and show its convergence result.

\subsection{Algorithm Design}
Before presenting our algorithm, we define the smoothed function of $f(x,y)$ as
$$H(x,y;z)=f(x,y)+\frac{p}{2}\|x-z\|^2,$$
for some auxiliary variable $z\in\mathbb{R}^{d_x}$. The squared term can introduce strong convexity and further smoothness in $x$ with a proper $p$. Then, we will perform GDA on the smoothed function $H(x,y;z)$, which is inspired by the first-order smoothed GDA (SGDA) in \cite{zhang2020single}. The \textit{zeroth-order block smoothed gradient descent ascent} (ZOB-SGDA) algorithm is proposed as shown in Algorithm \ref{alg:zoscgda}. Similarly to ZOB-GDA, we uniformly sample a block of dimensions $\mathcal{I}_k$ from $[d_x]$ for primal variables and only update the selected dimensions using BCGEs at each step. We denote the partial gradients along the dimensions $\mathcal{I}_k$ as
$$G_x^{\mathcal{I}_k}(x_k,y_k;z_k)=\sum_{i\in \mathcal{I}_k}\left( \frac{f(x_k+r_k e_i,y_k)-f(x_k,y_k)}{r_k} e_i + p(x_k-z_k)\odot e_i\right),$$
where $\odot$ denotes the Hadamard (element-wise) product. We also denote $G_x(x_k,y_k;z_k)=G_x^{\mathcal{I}}(x_k,y_k;z_k)$ when $\mathcal{I}=[d_x]$.
The update of $y_k$ follows the same way in ZOB-GDA. Additionally, an extra update for $z_k$ is introduced by an averaging step to stabilize the primal dynamics, which is crucial for improving ZOB-GDA. When $\gamma=1$, it is obvious that ZOB-SGDA reduces to ZOB-GDA.

\begin{algorithm}[htbp]
\caption{Zeroth-Order Block Smoothed Gradient Descent Ascent (ZOB-SGDA)}
\label{alg:zoscgda} 
\begin{algorithmic}[1]
    \STATE \textbf{Input:} Initial $(x_0,y_0)\in\mathbb{R}^{d_x}\times\mathcal{Y}, z_0=x_0$, maximum steps $K$, block size $b$, and the step sizes $\alpha, \beta, \gamma$.
    \FOR{$k=0,1,2,\cdots,K-1$}
    \STATE Uniformly sample $\mathcal{I}_k\subseteq [d_x]$ with $|\mathcal{I}_k|=b$ and update $x_k$ by 
    \begin{align}\label{eq:smoothed_primal_update}
    x_{k+1}=x_k-\alpha\cdot G^{\mathcal{I}_k}_x(x_k,y_k;z_k).
    \end{align}

    \STATE Update $y_k$ by $y_{k+1}=\mathcal{P}_{\mathcal{Y}}\left[y_k+\beta\cdot \nabla_y H(x_{k},y_k;z_k)\right].$
    \STATE Update $z_k$ by $z_{k+1}=\gamma x_{k+1}+(1-\gamma)z_k$.
    \ENDFOR
    \STATE \textbf{Output:} $\{(x_k,y_k)\}_{k=0}^K$
\end{algorithmic}
\end{algorithm}

\subsection{Convergence Analysis of ZOB-SGDA}
By properly setting the parameters for ZOB-SGDA, we can establish its convergence result in the sense of $\|\mathfrak{g}(x_k,y_k)\|$ as summarized in Theorem \ref{thm:zobsgda}.
\begin{theorem}\label{thm:zobsgda}
Suppose Assumptions \ref{assump:boundedness} and \ref{assump:smoothness} hold. The sequence $\{(x_k,y_k)\}_{k=0}^K$ is derived from ZOB-SGDA. Set the parameters $p\geq 3L$, $\sum_{k=0}^{K}r_k^2\leq\frac{1}{b}$, and $\alpha\leq \frac{1}{p+10L+1}$. Furthermore, let
$\beta\leq \min\big\{ \frac{1}{12L},\frac{\alpha^2(p-L)^2}{4L(\sqrt{N}+\alpha(p-L))^2} \big\}$, and $\gamma \leq \min\big\{ \sqrt{\frac{1}{KN}}, \frac{1}{36},\frac{1}{768p\beta} \big\}.$
Then, we have
$$\min_{k\leq K-1} \mathbb{E}\left[\|\mathfrak{g}(x_k,y_k)\|\right]\leq  \mathcal{O}\bigg(\left(\frac{N}{K}\right)^{1/4} \,\bigg).$$
\end{theorem}

The results in Theorem \ref{thm:zobsgda} show that ZOB-SGDA can converge to a stationary point at the rate of $\mathcal{O}\big(\big(\frac{N}{K}\big)^{\frac{1}{4}}\big)$, which corresponds to a KKT point under proper conditions (see Section \ref{sec:discussion} for further discussion).
Note that Theorem \ref{thm:zobsgda} does not require the Lipschitz continuity assumption for $f(x,y)$. We provide the proof of Theorem \ref{thm:zobsgda} in Appendix \ref{app:proof_of_thm_ZOBSGDA}.

Similarly, based on Theorem \ref{thm:zobsgda}, we have the following corollary to characterize the query complexity of ZOB-SGDA.

\begin{corollary}\label{coro:zobsgda}
Suppose that the conditions in Theorem \ref{thm:zobsgda} hold. For any sufficiently small $\epsilon>0$, set $\alpha, \beta, r_k$ as in Theorem \ref{thm:zobsgda} and $\gamma= \mathcal{O}(\epsilon^2/N)$. Then, the query complexity to achieve $\min_{k\leq K-1}\mathbb{E}\left[\left\| \mathfrak{g}(x_k,y_k)\right\|\right]\leq\epsilon$ is $\mathcal{O}\left(\frac{d_x}{\epsilon^4}\right)$.
\end{corollary}
We can see from Corollary \ref{coro:zobsgda} that ZOB-SGDA has the \emph{best-known} query complexity bound $\mathcal{O}(d_x/\epsilon^4)$ \citep{xu2024derivative}, regardless of the choice of block sizes. To the best of our knowledge, this result also matches the best-known complexity bound for first-order single-loop algorithms, up to a factor of $d_x$ \citep{xu2023unified,zhang2020single}.
In other words, ZOB-SGDA can achieve the best-known overall query complexity bound while maintaining controllable single-step efficiency. Specifically, only two queries are required for each step when we set $b=1$, which is more efficient than CGE-based algorithms that require $\mathcal{O}(d)$ queries for a gradient estimation \citep{hu2025convergence}.

\section{Discussions}\label{sec:discussion}
In this section, we present the following discussions regarding our algorithms and theoretical results.

First, block size plays an important role in determining the performance of our algorithms. Our theory shows that different block sizes share the same overall query complexity bound. Notably, block size directly governs the trade-off between single-step efficiency and iteration complexity. Smaller blocks reduce the number of queries per iteration but typically require more iterations to achieve certain errors. However, as shown in our following numerical results, properly tuned block sizes can substantially reduce the total number of queries in practice. One may therefore treat $b$ as a hyperparameter and use a few short, early-stopping runs to select a small but effective block size that balances per-iteration cost and overall query complexity. Developing finer-grained complexity bounds that make the dependence on $b$ explicit is an interesting direction for future work.

Second, under suitable conditions, the stationary points of (\ref{eq:general_minmax}) can provide solutions to (\ref{eq:constrained_problem}). Our theoretical results establish convergence guarantees to stationary points of $f(x,y)$ for the proposed algorithms, while our goal is to establish the convergence guarantees to the solutions to problem (\ref{eq:constrained_problem}). Therefore, we define the KKT gap at a point $(x,y)\in\mathbb{R}^{d_x}\times\mathcal Y$ as
\begin{align}\label{eq:kkt_gap}
\mathcal{K}(x,y)\;:=\;\left\| \nabla_x f(x,y)\right\|
\;+\;\max_{j\in\mathcal J}[c_j(x)]_+
\;+\;\max_{j\in\mathcal{J}} (y)_j |c_j(x)|,
\end{align}
where $[\cdot]_+:=\max (\cdot,0)$. The first term in \eqref{eq:kkt_gap} matches the primal component of $\|\mathfrak g(x,y)\|$ and measures the stationarity; the latter two terms quantify primal infeasibility and complementarity gap. It is obvious that $(x,y)$ is a KKT point of problem \eqref{eq:constrained_problem} when $\mathcal{K}(x,y)=0$. We provide the following lemma to characterize the relationship between $\|\mathfrak{g}(x,y)\|$ and the KKT gap, whose proof is provided in Appendix \ref{app:proof_of_lemma_stk}.
\begin{lemma}\label{lemma:stationary_to_kkt}
Suppose that $(x,y)\in\mathbb{R}^{d_x}\times\mathcal{Y}$ is any point satisfying $y< \overline{y}$ component-wise. Then, under Assumption \ref{assump:lipschitz}, there exists some constant $C>0$ such that $\mathcal{K}(x,y)\leq C\|\mathfrak{g}(x,y)\|$.
\end{lemma}

Lemma \ref{lemma:stationary_to_kkt} implies that the KKT gaps of our algorithms share the same convergence rate as the stationarity measure $\|\mathfrak{g}(x,y)\|$ under mild conditions.
The condition $y< \overline{y}$, in fact, requires that the iterates strictly stay below the upper bound of $\mathcal{Y}$. This is a common and fundamental limitation in the analysis of GDA-type algorithms \citep{nedic2009subgradient,liu2020min,xu2023unified}, which we believe is an open problem to be addressed in future work.

Moreover, we can extend block updates to broader problem settings. In our problem (\ref{eq:constrained_problem}), we deal with all constraints in the general form $c_j(x)\leq 0$. If equality constraints $c_j(x)=0$ have to be considered, we can incorporate them by adding two inequalities $-c_j(x)\leq 0$ and $c_j(x)\leq 0$. Besides, we can also consider some simple constraints by constraining the feasible space directly:
$$    \min_{x\in \mathcal{X}}\  h(x)\qquad \text{s.t.}\; c_j(x)\leq 0,\; \forall j\in\mathcal{J},$$
and deal with $x\in\mathcal{X}$ by projection, i.e., $x_{k+1}=\mathcal{P}_{\mathcal{X}}[x_k-\alpha G_x^{\mathcal{I}_k}(x_k,y_k)]$. We can establish the same theoretical results in our analysis when $\mathcal{X}$ is \textit{convex} and \textit{decomposable}, i.e., $\mathcal{X}=\prod_{i\in [d_x]} \mathcal{X}_i$. This requirement originates from the fundamental limit of block/coordinate updates \citep{lian2016comprehensive, jin2026zeroth}. Note that we need to make the modifications: $\mathfrak{g}_x(x,y)=\frac{1}{\alpha}\left(x-\mathcal{P}_\mathcal{X}[x-\alpha\nabla_x f(x,y)]\right)$, and $\Phi_{1/2L}(x) = \min_{u \in \mathcal{X}} \left\{\Phi(u) + L \|u - x\|^2 \right\}$ for the stationarity measure.
The extended theoretical results are straightforward to establish based on our analysis and the non-expansiveness of projection operators; thus, we omit the detailed analysis in this study.

Our methods can also be applied to stochastic settings, where $h(x)=\mathbb{E}[h(x;\xi)]$ and $c_j(x)=\mathbb{E}[c_j(x;\xi)]$, i.e., $\xi$ is a random variable defined on a probability space. However, we should not expect stronger convergence guarantees than those achieved by RGE-based algorithms, since the randomness of $\xi$ introduces additional variance into BCGEs, which weakens the benefit of low variance offered by our methods.

\section{Numerical Simulations}
In this section, we validate our algorithms through numerical experiments on an energy management problem and a high-dimensional parameter optimization problem. The detailed problem formulations and experimental settings are provided in Appendix D.

\subsection{Energy Management Problem}
We consider an energy management problem in a 141-bus distribution network with $d_x=168$ \citep{khodr2008maximum,zhou2025zeroth}. In this problem, the goal is to adjust the load of multiple users within a distribution network to curtail a specific amount of load while minimizing the cost of load adjustment. The problem is formulated as
\begin{align*}
    \min_{x\in \mathcal{X}} h(x)\qquad \text{s.t.}\; p_c(x)\leq D,
\end{align*}
where $x\in\mathbb{R}^{168}$ includes the active and reactive power loads of all buses. $h(x)$ is the objective function concerning operation cost and system safety; $p_c(x)$ is the total power exchange with the main grid after curtailment; $D$ is the target power level. Both the objective and constraint functions can only be computed via observations without analytical models.

\begin{figure}[htb]
    \centering
    \subfloat[]{   \includegraphics[scale=0.3]{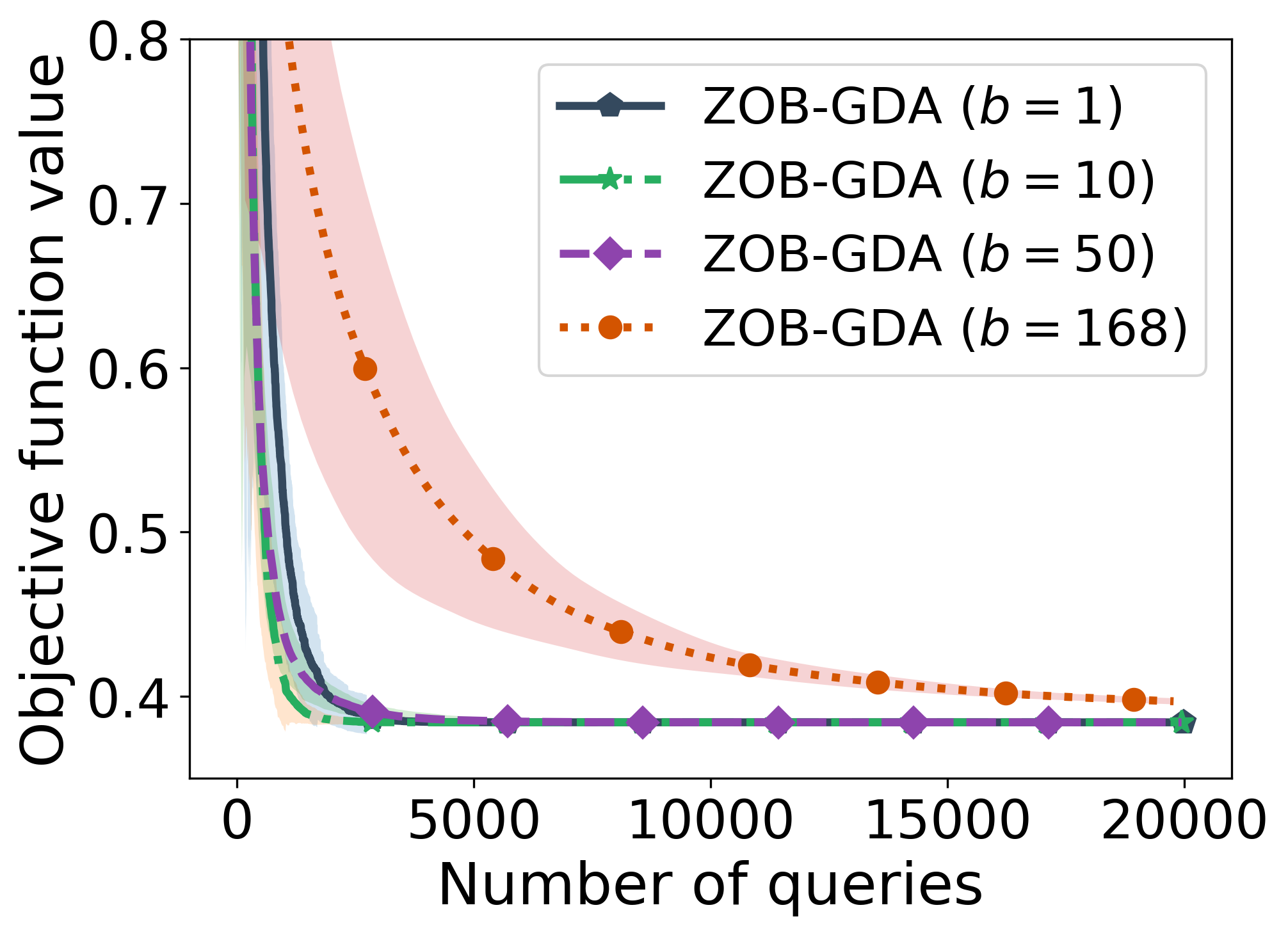}
          \label{figure nonconvex_obj_ZOBGDA}
	}
    \subfloat[]{   \includegraphics[scale=0.3]{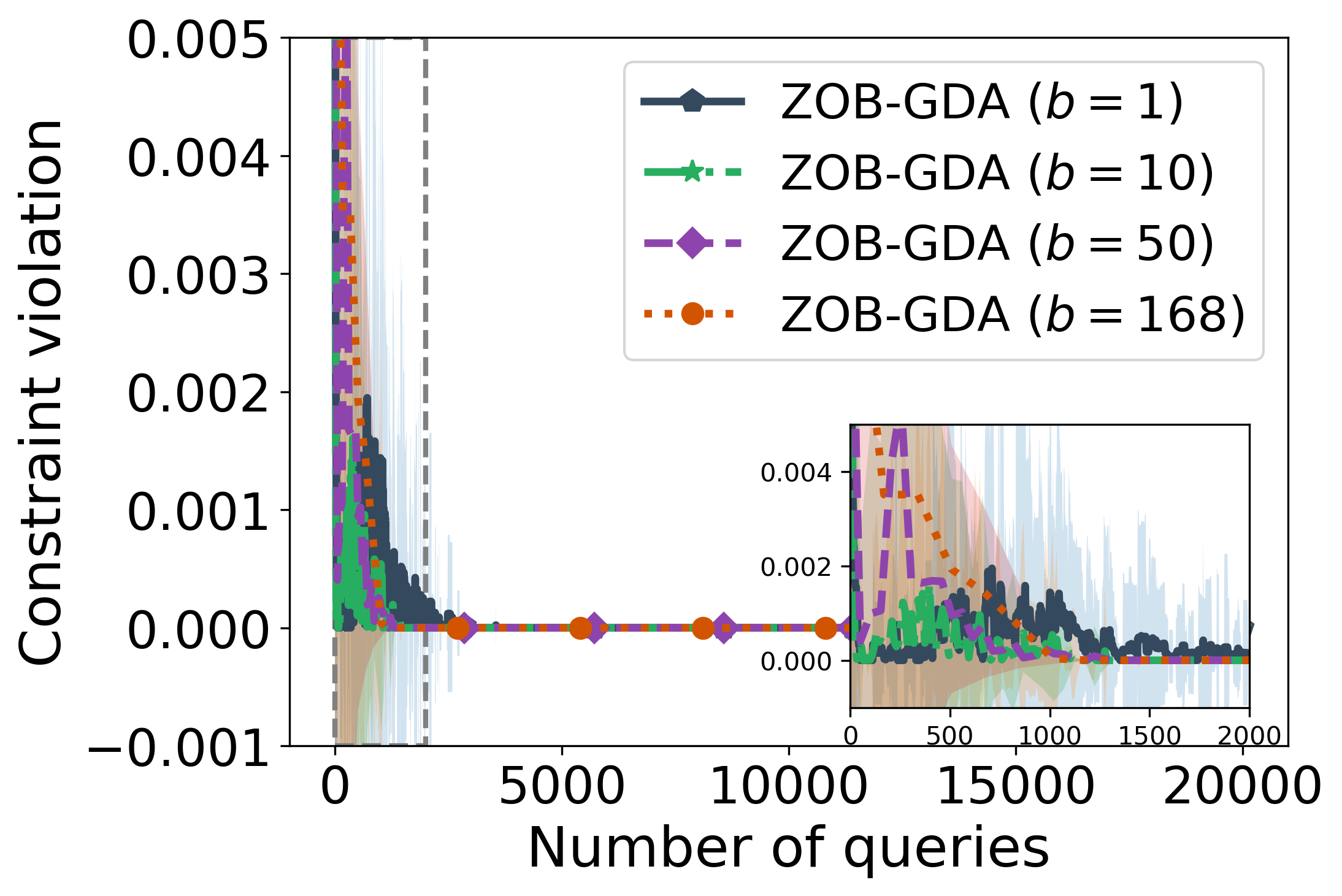}
          \label{figure nonconvex_vio_ZOBGDA}
	}
    \subfloat[]{   \includegraphics[scale=0.3]{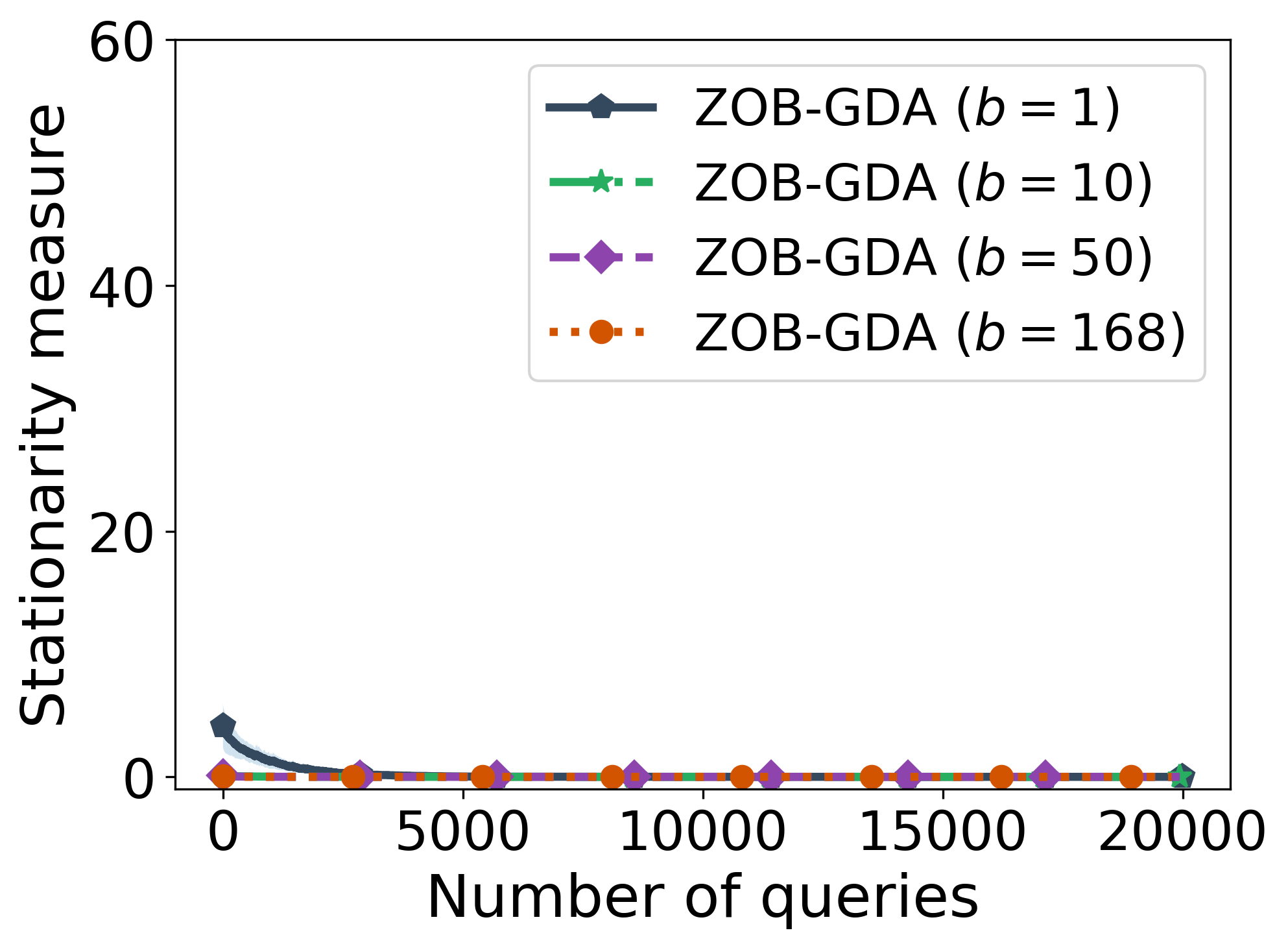}
          \label{figure nonconvex_sta_ZOBGDA}
	}
    \quad
    \subfloat[]{   \includegraphics[scale=0.3]{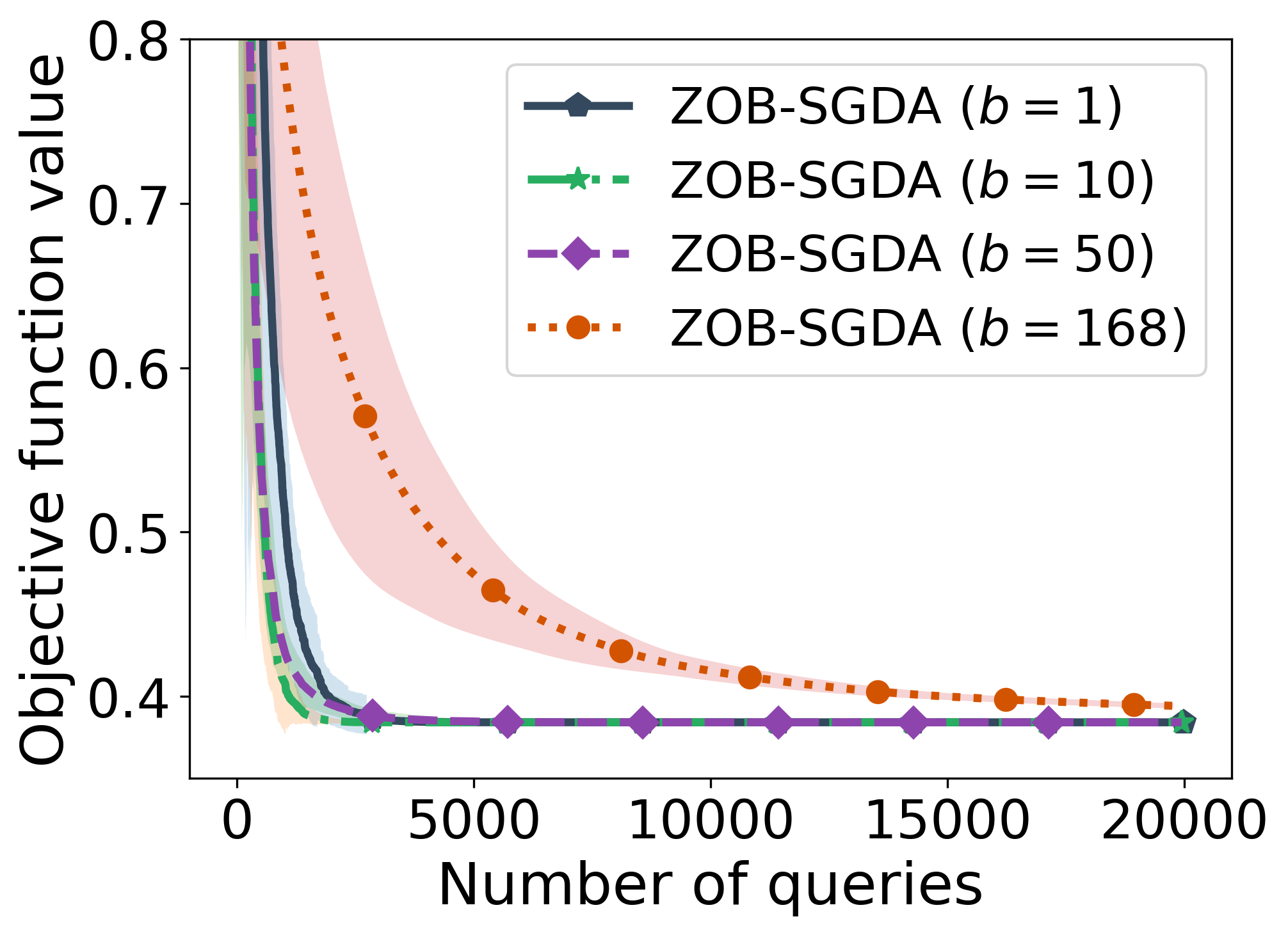}
          \label{figure nonconvex_obj_ZOB-SGDA}
	}
    \subfloat[]{   \includegraphics[scale=0.3]{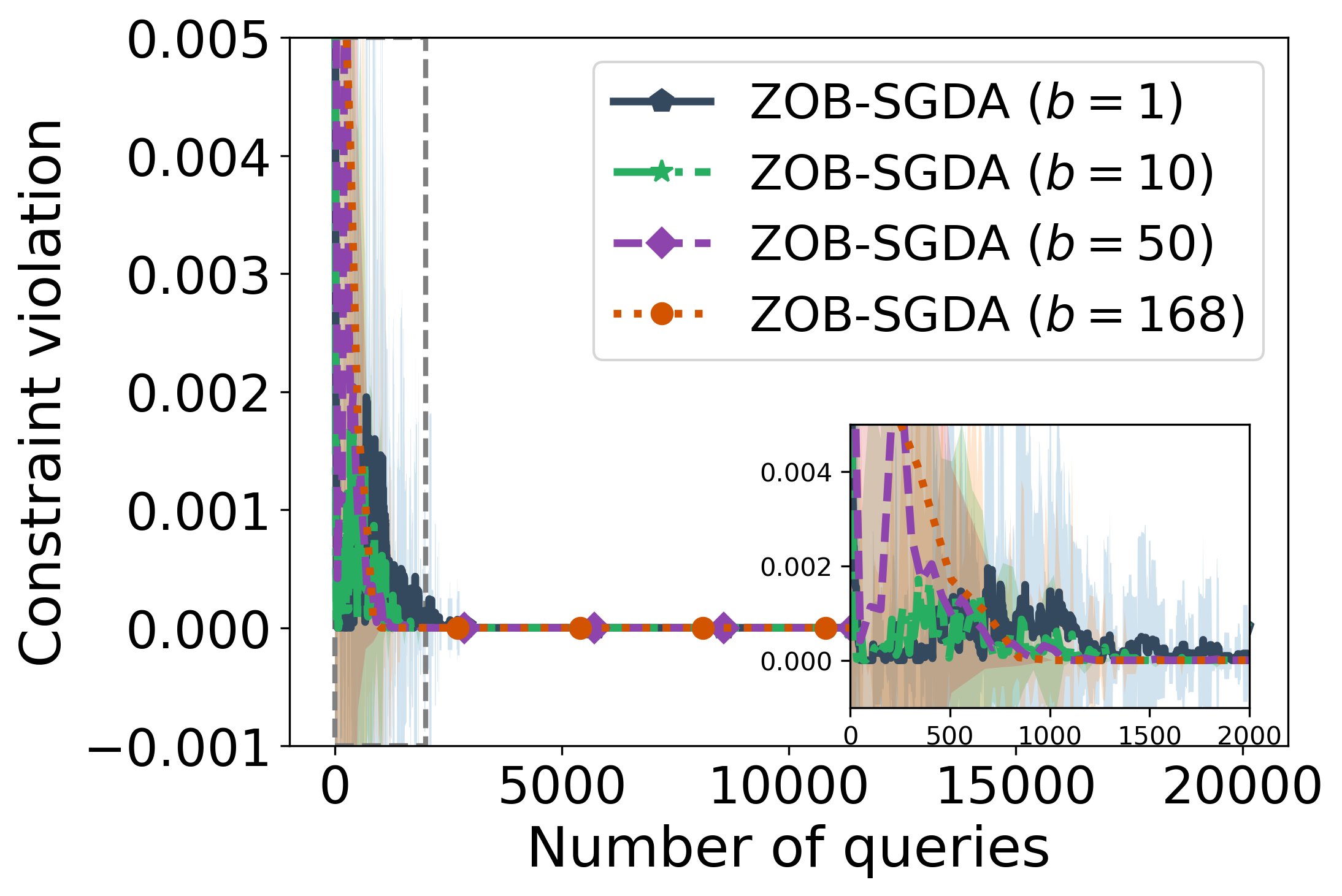}
          \label{figure nonconvex_vio_ZOB-SGDA}
	}
    \subfloat[]{   \includegraphics[scale=0.3]{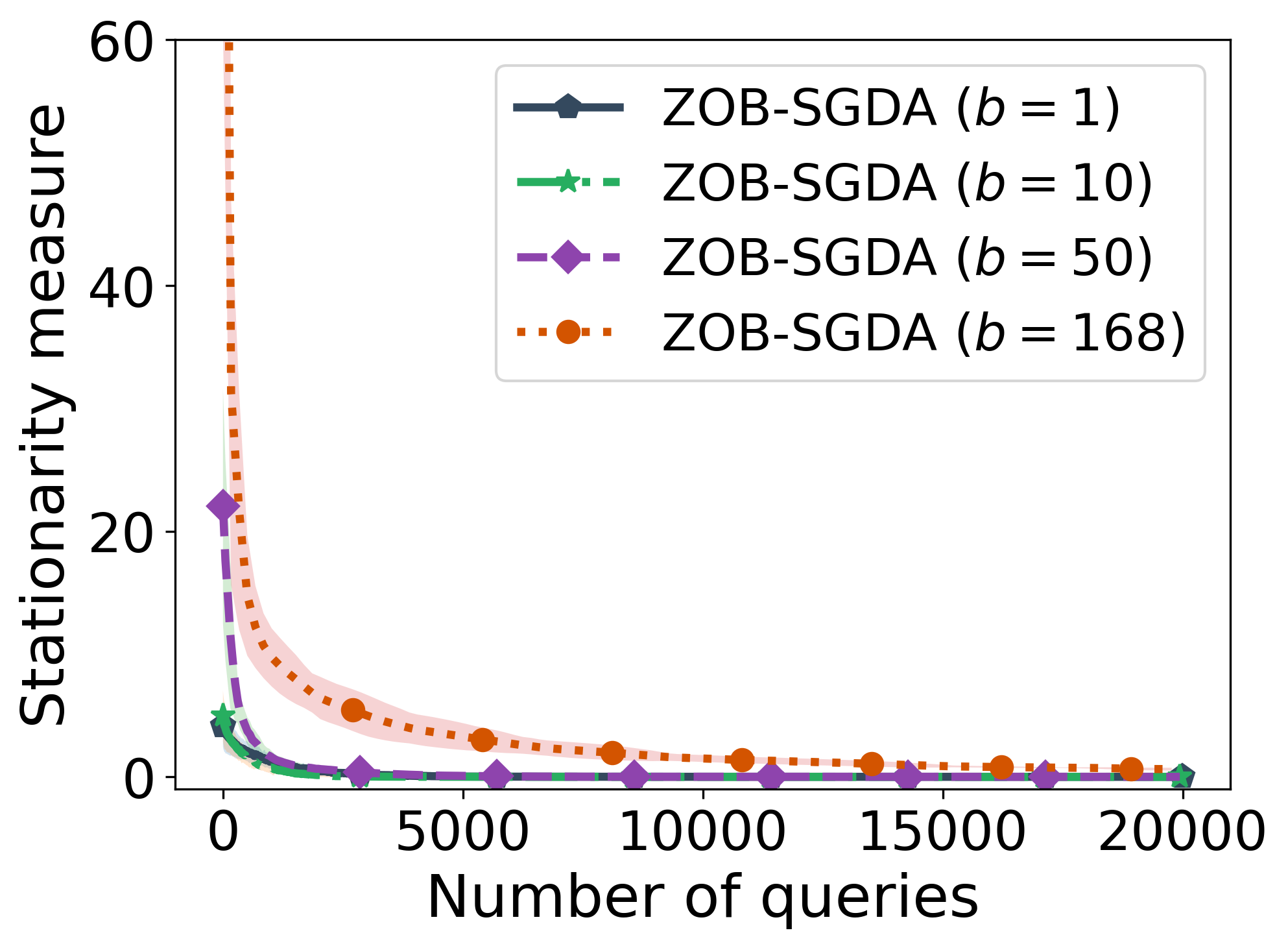}
          \label{figure nonconvex_sta_ZOB-SGDA}
	}
    \caption{Performance of ZOB-GDA and ZOB-SGDA in the energy management problem. (a), (b), and (c) present the objective function value, constraint violation, and stationarity measure of ZOB-GDA. (d), (e), and (f) present corresponding results for ZOB-SGDA}
    \label{fig:proposed}
\end{figure}

First, we apply ZOB-GDA and ZOB-SGDA to solve the problem using different block sizes. Their performance is averaged over 50 repeated runs with different initial points and is shown in Figure \ref{fig:proposed}. The dark curves represent the average performance, and the shaded areas represent the standard deviation. The results show that both ZOB-GDA and ZOB-SGDA can converge to the same objective function value with the constraint satisfied. Their stationarity measures can both converge to 0, which validates our theoretical guarantees. While different block sizes lead to convergence to the same objective function value, properly selecting the block sizes may substantially reduce the total number of queries required. Besides, we test ZOB-SGDA with/without parallel techniques (10 workers) for $b$ directional estimations to compare the wall-clock time to achieve different levels of relative errors (REs) and zero constraint violation (CV), as shown in Table \ref{tab:parallel_comparison}. The results indicate that parallelism only provides limited speedup when $b$ is large and, in many cases, even leads to worse wall-clock performance for ZOB-SGDA. It may stem from the fact that each function evaluation is inexpensive, and the additional time for parallel processing offsets the saved time due to parallel computing.
\begin{table}[htbp]
\centering
\setlength{\abovecaptionskip}{4pt}
\setlength{\belowcaptionskip}{4pt}
\caption{Computational Time of ZOB-SGDA under Parallel and Non-parallel Settings (Unit: seconds)}
\label{tab:parallel_comparison}
\renewcommand{\arraystretch}{1.05}
\setlength{\tabcolsep}{4pt}

\begin{tabular}{c cc cc cc cc}
\toprule
& \multicolumn{2}{c}{$b=1$}
& \multicolumn{2}{c}{$b=10$}
& \multicolumn{2}{c}{$b=50$}
& \multicolumn{2}{c}{$b=168$} \\
\cmidrule(lr){2-3}\cmidrule(lr){4-5}\cmidrule(lr){6-7}\cmidrule(lr){8-9}
RE
& Parallel & Non-parallel
& Parallel & Non-parallel
& Parallel & Non-parallel
& Parallel & Non-parallel \\
\midrule
10\%  & 19.38 & 6.99  & 4.19 & 1.12          & 3.34 & \textbf{0.71} & 4.84  & 4.95 \\
1\%   & 31.88 & 11.82 & 5.87 & \textbf{1.75} & 6.06 & 1.94          & 17.04 & 17.28 \\
0.1\% & 42.57 & 15.89 & 6.93 & \textbf{2.22} & 7.43 & 2.66          & 22.34 & 24.92 \\
\bottomrule
\end{tabular}
\end{table}

\begin{figure}[htbp]
    \centering
    \subfloat[]{   \includegraphics[scale=0.3]{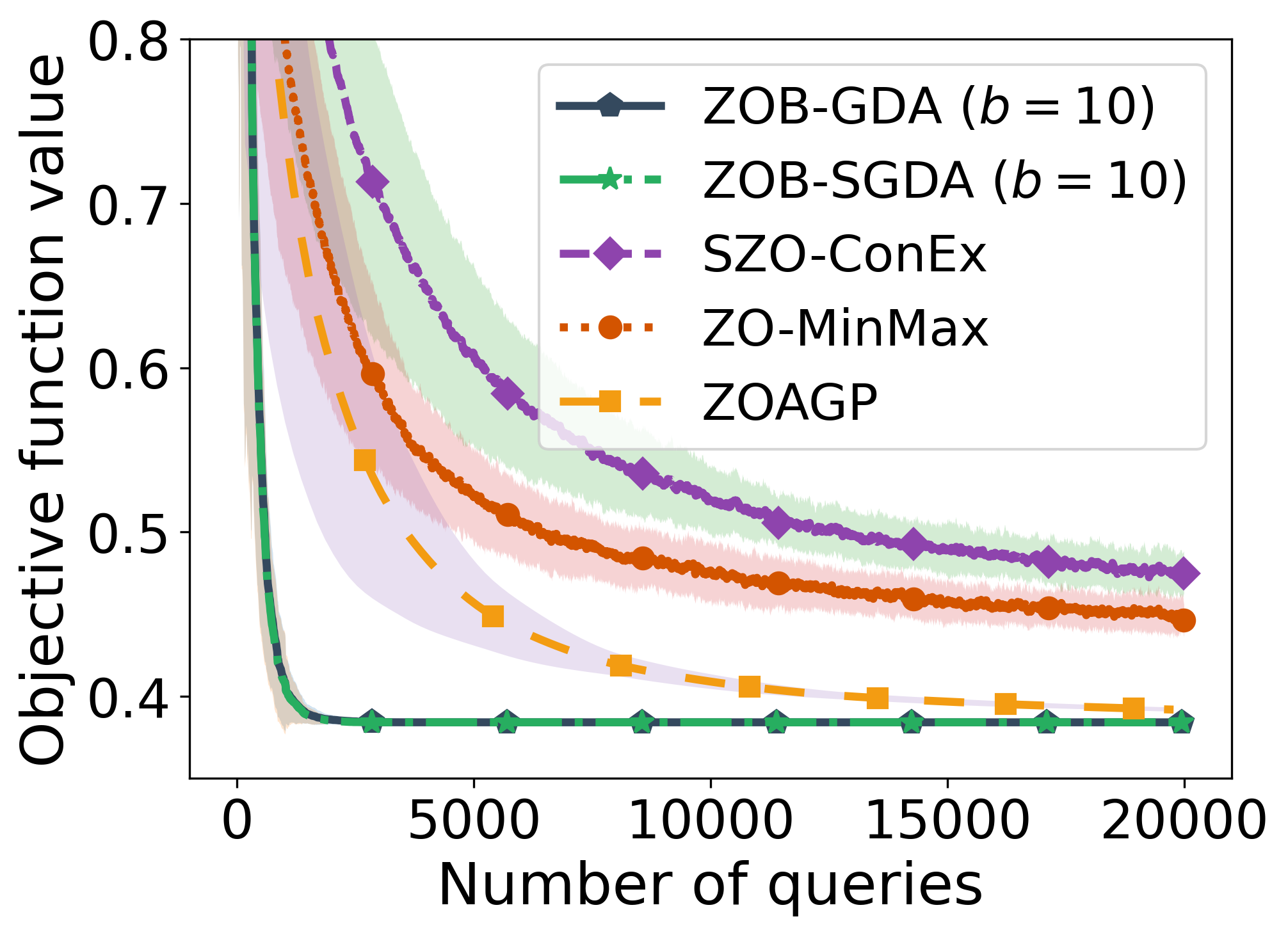}
          \label{figure nonconvex_obj_benchmark}
	}
    \subfloat[]{   \includegraphics[scale=0.3]{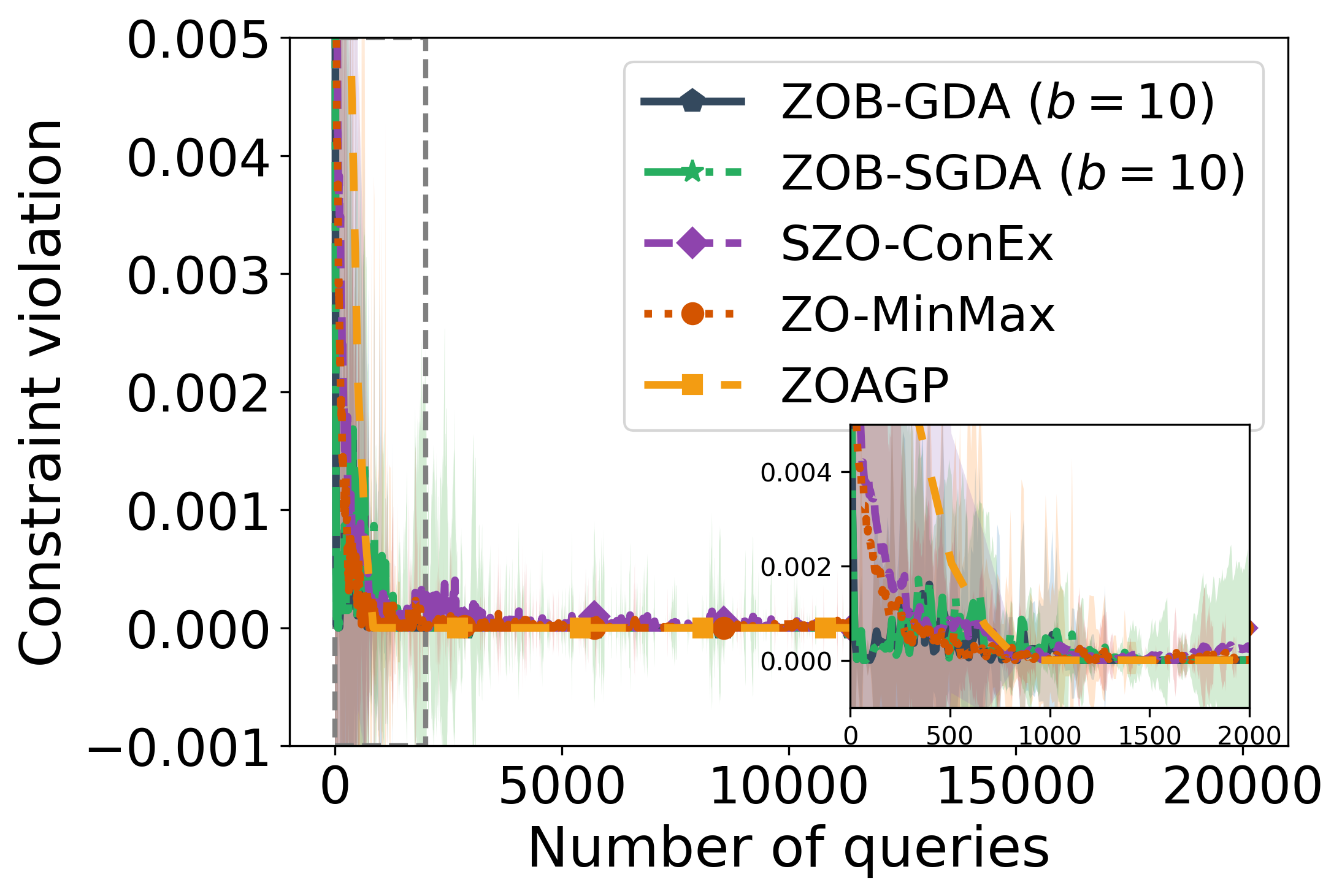}
          \label{figure nonconvex_vio_benchmark}
	}
    \subfloat[]{   \includegraphics[scale=0.3]{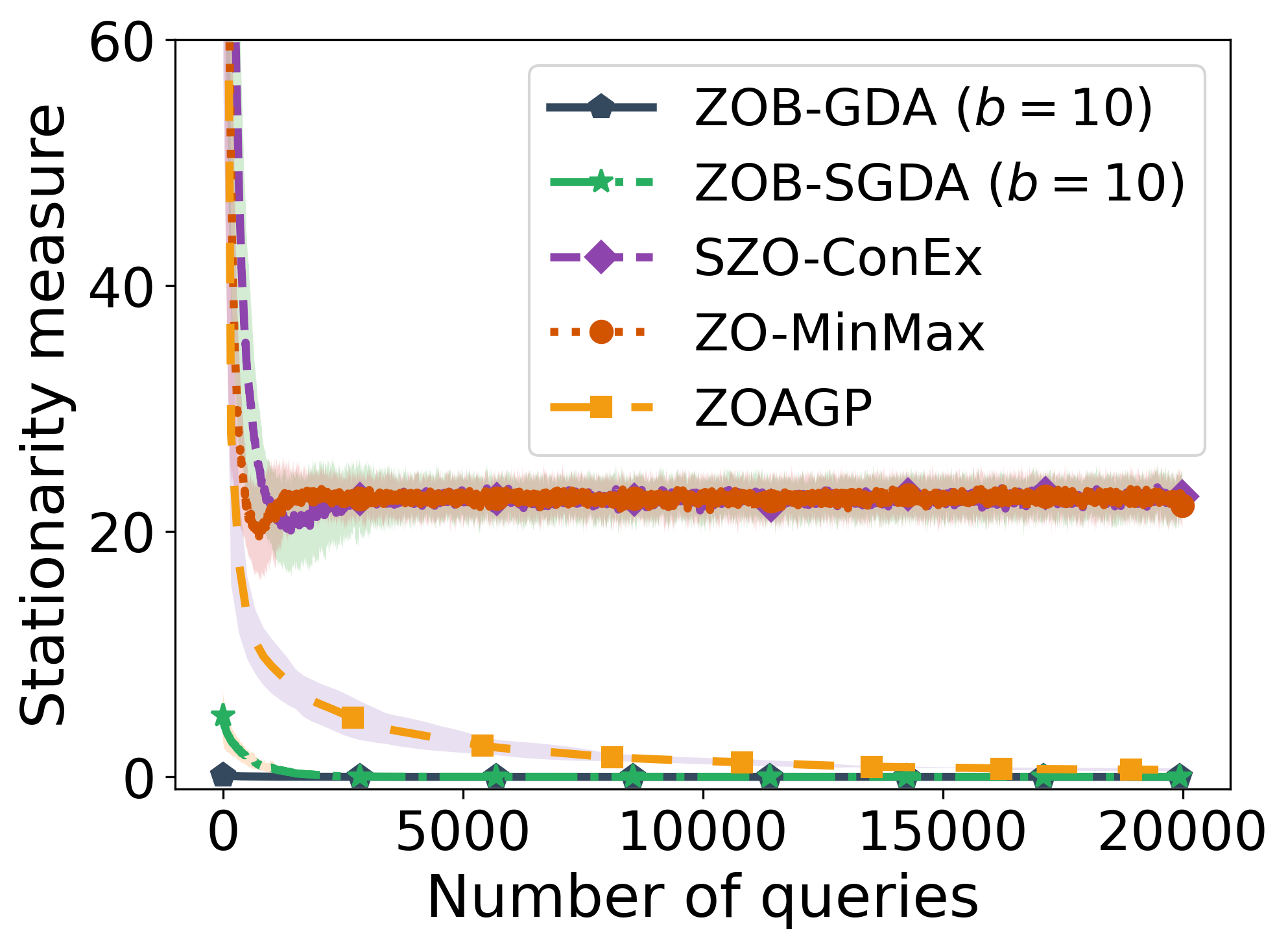}
          \label{figure nonconvex_sta_benchmark}
	}
    \caption{Performance comparison of different algorithms in the energy management problem}
    \label{fig:compare}
\end{figure}

\begin{table}[htbp]
  \centering
  \caption{Average numbers of iterations and queries required to generate certain solutions in the energy management problem (``NaN'' means no runs can achieve such a solution)}
  \label{tab:compare}

  \setlength{\tabcolsep}{3pt}
  \renewcommand{\arraystretch}{1.05}

  \begin{tabular}{l c cc cc cc}
    \toprule
    \multicolumn{2}{c}{RE}
    & \multicolumn{2}{c}{10\%}
    & \multicolumn{2}{c}{1\%}
    & \multicolumn{2}{c}{0.1\%} \\
    \cmidrule(lr){3-4}\cmidrule(lr){5-6}\cmidrule(lr){7-8}
    \multicolumn{2}{c}{}
    & Iteration & Queries
    & Iteration & Queries
    & Iteration & Queries \\
    \midrule

    \multirow{4}{*}{ZOB-GDA}
      & $b=1$   & 722.06  & 1444.12  & 1213.48  & 2426.96  & 1584.84  & 3169.68 \\
      & $b=10$  & 75.00   & 825.00   & 138.08   & 1518.88  & 182.00   & 2002.00 \\
      & $b=50$  & 24.32 & 1240.32  & 68.02 & 3469.02  & 92.24 & 4704.24 \\
      & $b=168$ & 60.06   & 10150.14  & 216.68   & 36618.92 & 285.90   & 48317.10 \\
    \addlinespace[2pt]

    \multirow{4}{*}{ZOB-SGDA}
      & $b=1$   & 722.80  & 1445.60  & 1195.94  & 2391.88  & 1594.32  & 3188.64 \\
      & $b=10$  & 74.00   & \textbf{814.00}   & 131.90   & \textbf{1450.90}  & 169.68   & \textbf{1866.48}\\
      & $b=50$  & \textbf{22.18}  & 1131.18  & \textbf{58.26}    & 2971.26  & \textbf{77.76}    & 3965.76 \\
      & $b=168$ & 52.06   & 8798.14  & 187.42   & 31673.98 & 266.59   & 45054.15 \\
    \addlinespace[2pt]
    \midrule

    ZO-MinMax &  & 12535.10 & 25070.20 & 12771.33 & 25542.65 & NaN & NaN \\
    SZO-ConEx &  & 12817.62 & 51270.49 & NaN & NaN & NaN & NaN \\
    ZOAGP     &  & 45.14 & 7628.66 & 162.82 & 27516.58 & 255.96 & 43257.24 \\
    \bottomrule
  \end{tabular}
\end{table}

We also compare our algorithms (block size $b=10$) with three others, i.e., ZO-MinMax \citep{liu2020min}, SZO-ConEx \citep{nguyen2023stochastic}, and ZOAGP \citep{xu2024derivative}, which can be applied to solve problem (\ref{eq:constrained_problem}). Each algorithm is tested with 50 runs, and the average performance is presented in Figure \ref{fig:compare}. The results show that our algorithms can both converge to a solution significantly faster than others. To better compare their query complexities, we present the average number of queries required to generate solutions with zero CV and different REs, as shown in Table \ref{tab:compare}. ZOB-GDA and ZOB-SGDA exhibit highly similar performance under different block sizes, while the complexity bound of the latter is theoretically tighter. Notably, proper block sizes may lead to much improved query complexities (about 10 times better than existing methods). Moreover, SZO-ConEx and ZO-MinMax have much worse performance due to the large variance of RGEs.

Based on the numerical results, we highlight two key observations. First, although all choices of block size $b$ share the same query complexity bound, their empirical convergence behaviors can differ substantially. This is because our query complexity bounds provide only worst-case guarantees, whereas in practice, the algorithms may perform significantly better with an appropriately chosen $b$. We recommend starting with a small block size and tuning by early stopping runs until a good empirical trade-off between single-step cost and convergence speed is observed. Second, ZOB-GDA exhibits performance comparable to that of ZOB-SGDA in our experiments, despite the fact that ZOB-SGDA enjoys a strictly better query complexity bound. This observation again suggests that the theoretical results are conservative and may not fully capture the typical behavior encountered in practice. The situation is similar to the first-order setting, where GDA is theoretically dominated by its variants, such as extra-gradient methods, yet often performs competitively and can enjoy strong empirical performance except in certain pathological cases.


\subsection{High-Dimensional Parameter Optimization}
We further test our algorithms in a 1000-dimensional parameter optimization problem to validate the capability of our methods for scalable problems. We consider the following parameter optimization problem:
\begin{align*}
\min_{x\in\mathbb{R}^{1000}} &h(x)=0.5\cdot \|Bx\|^2+0.1\cdot\sum_{i}x^4(i)\\
\text{s.t.}\ & c(x)=(1+e^{-q^Tx})^{-1}-\Bar{c}\leq 0,
\end{align*}
where $B\in\mathbb{R}^{1000\times 1000}$ is a matrix, $q\in\mathbb{R}^{1000}$ is a vector, and $\Bar{c}$ is a scalar. The detailed parameter settings for the problem and algorithms are provided in Appendix D. By the early-stopping rule, we set the block size $b=30$ for our algorithms, which yields good performance.

\begin{figure}[htbp]
    \centering
    \subfloat[]{   \includegraphics[scale=0.3]{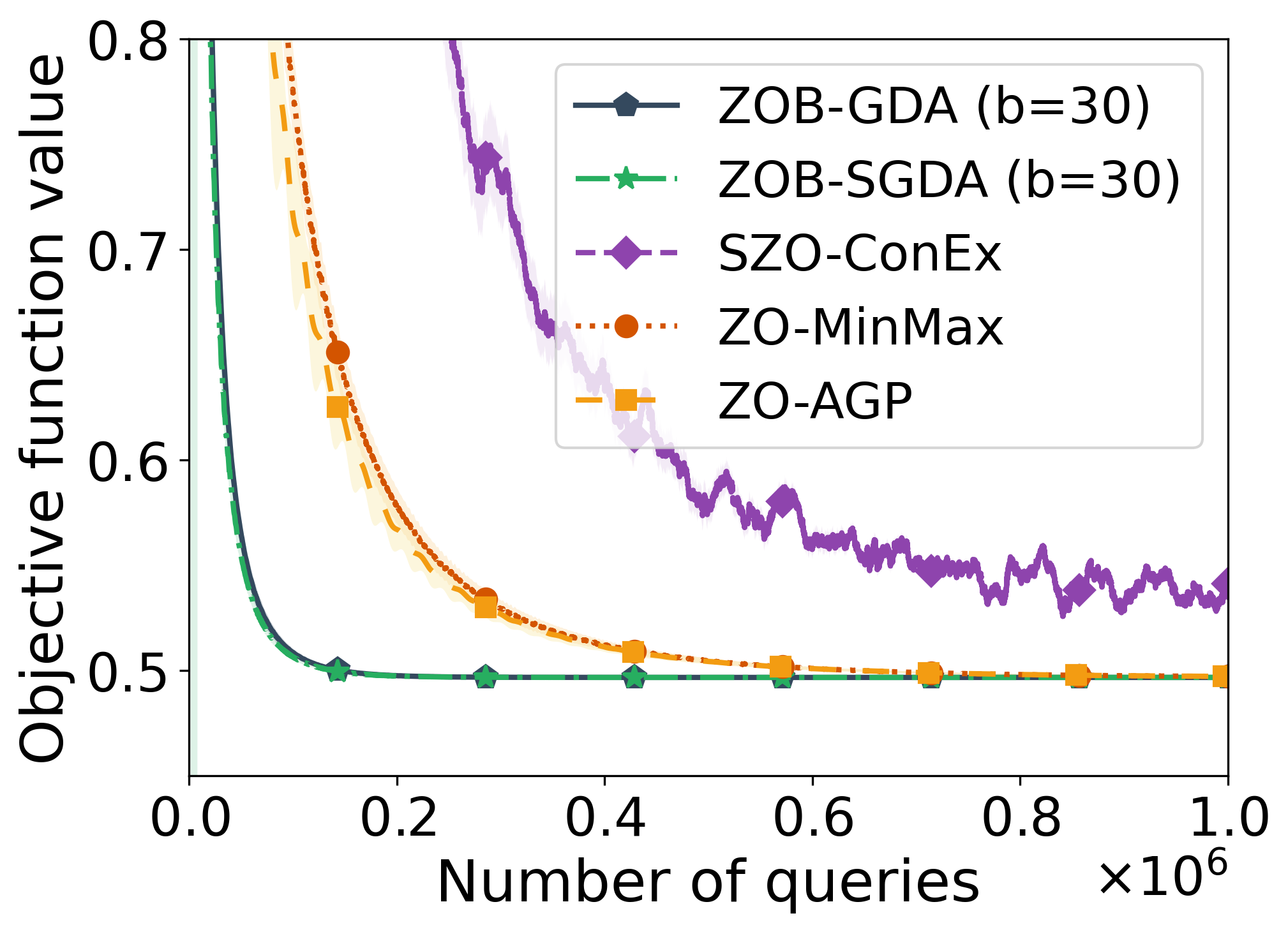}
          \label{figure nonconvex_obj_benchmark_high}
	}
    \subfloat[]{   \includegraphics[scale=0.3]{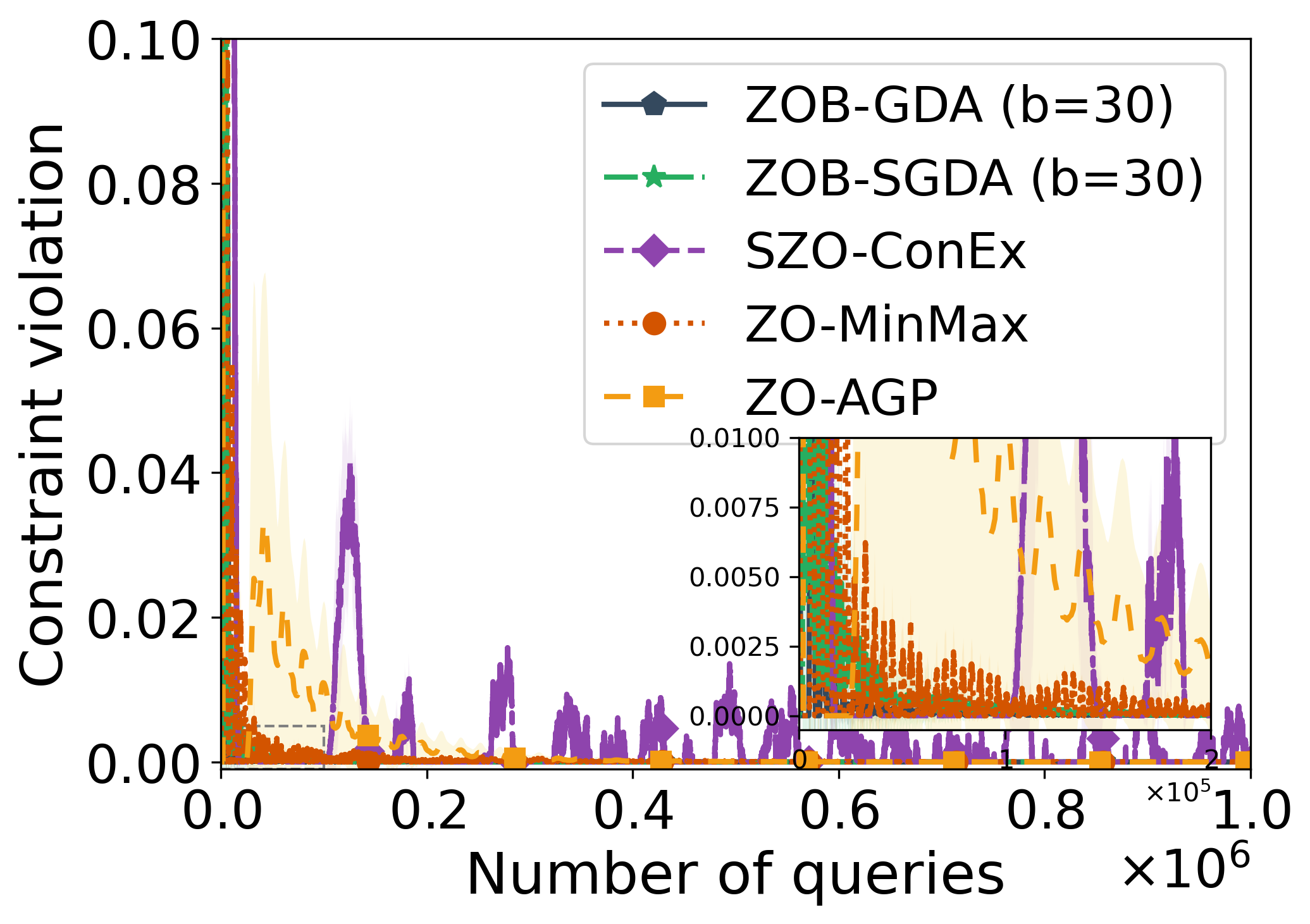}
          \label{figure nonconvex_vio_benchmark_high}
	}
    \subfloat[]{   \includegraphics[scale=0.3]{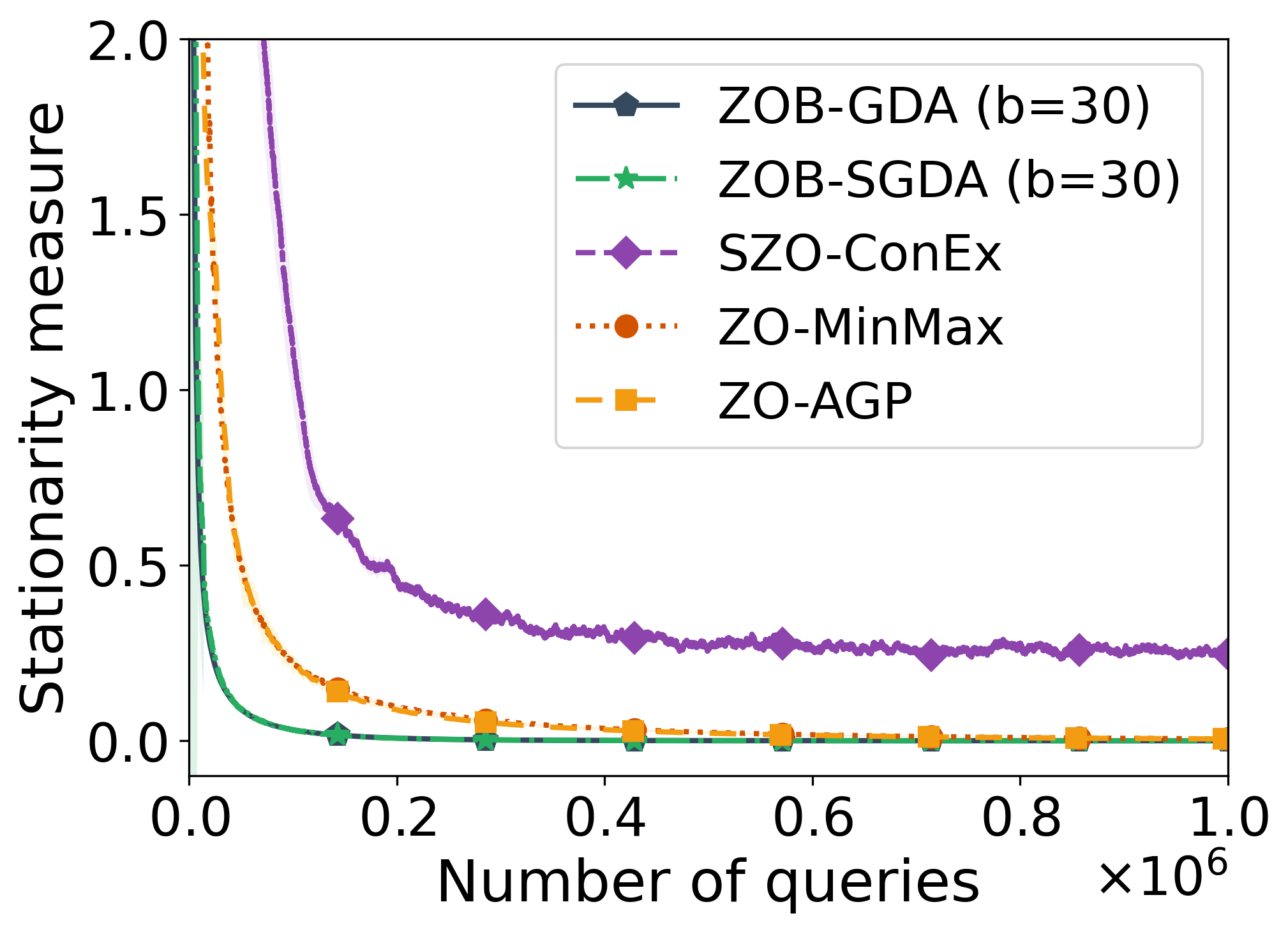}
          \label{figure nonconvex_sta_benchmark_high}
	}
    \caption{Performance comparison of different algorithms in the parameter optimization problem}
    \label{fig:high_compare}

\end{figure}

\begin{table}[htbp]
  \centering
  \caption{Average numbers of queries required to generate certain solutions in the parameter optimization problem (``NaN'' means no runs can achieve such a solution)}
  \label{tab:high_compare}

  \setlength{\tabcolsep}{4pt}
  \renewcommand{\arraystretch}{1.05}

  \begin{tabular}{l cc cc cc}
    \toprule
    RE / CV
      & \multicolumn{2}{c}{10\% / 0.1}
      & \multicolumn{2}{c}{1\% / 0.01}
      & \multicolumn{2}{c}{0.1\% / 0.001} \\
    \cmidrule(lr){2-3}\cmidrule(lr){4-5}\cmidrule(lr){6-7}
      & Iteration & Queries
      & Iteration & Queries
      & Iteration & Queries \\
    \midrule
    ZOB-GDA ($b=30$)  & 1853.00  & 57443.00  & 4081.70  & 126532.70  & 6321.30 & 195960.30 \\
    ZOB-SGDA ($b=30$) & 1704.10  & \textbf{52827.10}  &  3795.55  & \textbf{117662.05}  & 5876.90 & \textbf{182183.90} \\
    \addlinespace[2pt]
    \midrule
    ZO-MinMax & 125127.85 &  250255.70  & 279401.10 & 558802.20 & 436177.55 & 872355.10  \\
    SZO-ConEx & 172631.15 & 690525.60 & NaN & NaN & NaN & NaN \\
    ZOAGP     & \textbf{234.00}     &  234234.00  & \textbf{556.45} & 557006.45 & \textbf{887.30} & 888187.30 \\
    \bottomrule
  \end{tabular}
\end{table}



The results are summarized in Figure \ref{fig:high_compare} and Table \ref{tab:high_compare}. Except for SZO-ConEx, all algorithms can generate solutions that meet the prescribed levels of RE and CV, whereas our algorithms converge substantially faster in terms of query efficiency. Notably, the performance of ZOB-SGDA is slightly better than ZOB-GDA in number of queries. In contrast, SZO-ConEx fails to generate high-accuracy solutions, and its performance is significantly worse than that of our algorithms. Interestingly, although ZOAGP and ZO-MinMax apply CGEs and RGEs, respectively, they exhibit similar performance. This suggests that the CGE-based algorithm does not necessarily converge faster in high-dimensional cases, even though CGEs typically enjoy lower variance.

\section{Conclusion}\label{sec:Conclusion}
In this research, we study a general optimization problem with black-box constraints. We reformulate it as a min-max problem, and then apply zeroth-order optimization (ZO) methods to solve it using the input-output information. Specifically, by integrating block updates with gradient descent ascent (GDA), we develop two novel algorithms, called ZOB-GDA and ZOB-SGDA, which achieve efficiency in both single-step gradient estimation and the overall query complexity. Our theoretical results demonstrate that ZOB-GDA achieves the same query complexity bound as its first-order counterpart with an additional dimension-dependent factor, and ZOB-SGDA enjoys the best-known complexity bound. In addition, our numerical experiments validate their superior performance. However, our work on block updates in constrained ZO is just a beginning. There are still several open challenges. First, while the block update framework is a broadly applicable technique for improving single-step efficiency, its integration with other primal–dual algorithms requires more study. Second, although we anticipate that the benefits of block updates in stochastic constrained ZO will be more limited than in deterministic settings, rigorous validation requires further investigation.

\bibliographystyle{apalike}
\bibliography{reference}

\newpage

\begin{center}
    {\LARGE\bfseries Appendices}
\end{center}

\appendix

\section{Proof of Theorem \ref{thm:ZOB-GDA}}\label{app:proof_of_thm_ZOBGDA}
First, define the proximity operator $$\text{prox}_{\lambda h}(x)=\arg\min_{u\in \mathbb{R}^{d_x}}\left\{h(u)+\frac{1}{2\lambda}\|x-u\|^2\right\}.$$
Define the filtration: $\mathcal{F}_k=\sigma(x_0, y_0, \mathcal{I}_0, \cdots, \mathcal{I}_{k-1},x_k,y_k)$.
Then, we provide the following lemma to bound the one-step drift of $\Phi_{1/2L}(x_k)$.
\begin{lemma}\label{lemma:lemma1}
    Let $\Delta_k=\Phi (x_k)-f(x_k,y_k)$. The following inequality holds for any $k\geq 0$,
    \begin{align}\label{eq:lemma1}
    &\mathbb{E}\left[\left. \Phi_{1/2L}(x_{k+1})-\Phi_{1/2L}(x_k) \right| \mathcal{F}_k\right]\nonumber\\
    \leq & \frac{2\alpha L}{N}\Delta_k -\frac{\alpha}{8N}\left\|\nabla\Phi_{1/2L}(x_k)\right\|^2 +\frac{2\alpha^2\Lambda^2L}{N} +\frac{\alpha^2 b L^3r_k^2}{2}+\frac{\alpha b L^2r_k^2}{2}.
    \end{align}
\end{lemma}
The proof of Lemma \ref{lemma:lemma1} is delayed in Appendix \ref{app:proof_of_lemma1}. We further provide the following lemma to bound the summation of $\Delta_k$.
\begin{lemma}\label{lemma:lemma3}
For any integer $B$ that can divide $K$, we have
\begin{align*}
\frac{1}{K}\sum_{k=0}^{K-1}\Delta_k\leq & \alpha\Lambda_0^2(B+1)+\frac{R_y^2}{2\beta B}+\frac{\Delta_0}{K}.
\end{align*}
\end{lemma}
The proof of Lemma \ref{lemma:lemma3} is provided in Appendix \ref{app:proof_of_lemma3}. Taking the telescoping sum of (\ref{eq:lemma1}) and taking the total expectation, we have
\begin{align*}
&\frac{1}{K}\sum_{k=0}^{K-1}\mathbb{E}\left[\left\| \nabla \Phi_{1/2L}(x_k)\right\|^2\right]\\
\leq & \frac{8N}{\alpha K}\mathbb{E}\left[ \Phi_{1/2L}(x_{0})-\Phi_{1/2L}(x_{K}) \right]+\frac{16 L}{K}\sum_{k=0}^{K-1}\mathbb{E}\left[\Delta_k\right] \\
& + 16\alpha \Lambda^2L + \frac{4\alpha d_xL^3}{K}\sum_{k=1}^{K-1}r_k^2 + \frac{4d_x L^2}{K} \sum_{k=0}^{K-1}r_k^2\\
\leq\; & \frac{8N\Delta_\Phi}{\alpha K}+\frac{16 L}{K}\sum_{k=0}^{K-1}\mathbb{E}\left[\Delta_k\right]+ 16\alpha \Lambda^2L +\frac{8d_x L^2}{K} \sum_{k=0}^{K-1}r_k^2,
\end{align*}
where $\Delta_\Phi:=\Phi_{1/2L}(x_0)-\min_{x\in\mathbb{R}^{d_x}}\Phi_{1/2L}(x)$. The last step follows from the fact that $\alpha L\leq 1$. Then, we can combine the above inequality with the result in Lemma \ref{lemma:lemma3} to get
\begin{equation}
\begin{aligned}\label{eq:prop2_result}
&\frac{1}{K}\sum_{k=0}^{K-1}\mathbb{E}\left[\left\| \nabla \Phi_{1/2L}(x_k)\right\|^2\right]\\
\leq & \frac{8N\Delta_\Phi}{\alpha K}+\frac{16L\Delta_0}{K}+\frac{8NL^2}{K}+\frac{8LR_y^2}{\beta B}+16\alpha L\Lambda_0^2(B+2),
\end{aligned}
\end{equation}
Due to that, for any $x\in\mathbb{R}^{d_x}$,
\begin{align*}
\Phi_{1/2L}(x)=\; &\min_{u\in\mathbb{R}^{d_x}}\{\Phi (u)+L\|x-u\|^2\}\\
=\; &\min_{u\in\mathbb{R}^{d_x}}\left\{\max_{y\in\mathcal{Y}} f(u,y)+L\|x-u\|^2\right\}\\
\geq\; & \min_{u\in\mathbb{R}^{d_x}}\left\{\underline{f}+L\|x-u\|^2\right\}\\
=\; & \underline{f},
\end{align*}
we have $\Delta_{\Phi}$ is upper bounded. Without loss of generality, set $B=\frac{R_y}{\Lambda_0}\sqrt{\frac{1}{2\alpha \beta}}$ that can divide $K$, then we can derive 
$$\frac{1}{K}\sum_{k=0}^{K-1}\mathbb{E}\left[\left\| \nabla \Phi_{1/2L}(x_k)\right\|^2\right]\leq \mathcal{O}\left(\frac{N}{\alpha K}\right)+\epsilon_c^2,$$ 
which leads to $$\min_{k\leq K-1}\mathbb{E}\left[\left\| \nabla\Phi_{1/2L}(x_k)\right\|\right]\leq \left(\mathcal{O}\left(\frac{N}{\alpha K}\right)+\epsilon_c^2\right)^{1/2}\leq \mathcal{O}\left(\sqrt{\frac{N}{\alpha K}}\right)+\epsilon_c.$$

\subsection{Proof of Lemma \ref{lemma:lemma1}}\label{app:proof_of_lemma1}
Denote $\hat{x}_k=\text{prox}_{\Phi/2L}(x_k)$. 
Using the definition of $\Phi_{1/2L}(x_{k+1})$, we have
\begin{align}\label{eq:lemma3_1}
\Phi_{1/2L}(x_{k+1})\leq \Phi (\hat{x}_k)+L\|\hat{x}_k-x_{k+1}\|^2.
\end{align}
Based on the update of $x_k$, we have
\begin{equation}
\begin{aligned}\label{eq:primal_update_decomp}
& \|\hat{x}_k- x_{k+1} \|^2\\
=\;&  \left \| \hat{x}_k-x_k+\alpha G_x^{\mathcal{I}_k} (x_k,y_k) \right\|^2\\
=\;& \|\hat{x}_k-x_k\|^2+2\alpha \langle G_x^{\mathcal{I}_k}(x_k,y_k), \; \hat{x}_k-x_{k} \rangle +\alpha^2 \left\|G_x^{\mathcal{I}_k}(x_k,y_k)\right\|^2.
\end{aligned}
\end{equation}

Substituting the above equation into (\ref{eq:lemma3_1}) leads to
\begin{equation}
\begin{aligned}\label{eq:lemma3_2}
&\Phi_{1/2L}(x_{k+1})\\
\leq\; & \Phi(\hat{x}_k)+L\|\hat{x}_k-x_k\|^2+\alpha^2L\left\|G_x^{\mathcal{I}_k}(x_k,y_k)\right\|^2+ 2\alpha L \langle G_x^{\mathcal{I}_k}(x_k,y_k),\; \hat{x}_k-x_k \rangle \\
=\;&\Phi_{1/2L}(x_{k}) +\alpha^2L\left\|G_x^{\mathcal{I}_k}(x_k,y_k)\right\|^2+ 2\alpha L \langle G_x^{\mathcal{I}_k}(x_k,y_k),\; \hat{x}_k-x_k \rangle. 
\end{aligned}
\end{equation}
Taking the conditional expectation of the third term on the right-hand side of (\ref{eq:lemma3_2}), we have
\begin{align}\label{eq:lemma3_3}
&\mathbb{E}\left[\left. 2\alpha L \langle G_x^{\mathcal{I}_k}(x_k,y_k),\; \hat{x}_k -x_k \rangle \right| \mathcal{F}_k\right]\nonumber\\
=&\frac{2\alpha L}{N} \langle \nabla_{x}f(x_k,y_k), \hat{x}_k-x_k \rangle  +  \frac{2\alpha L}{N} \langle G_x(x_k,y_k)-\nabla_x f(x_k,y_k),\; \hat{x}_k -x_k \rangle\nonumber\\
\leq & \frac{2\alpha L}{N} (f(\hat{x}_k,y_k)-f(x_k,y_k))+\frac{\alpha L^2}{N}\|\hat{x}_k-x_k\|^2\nonumber\\
& +\frac{2\alpha L}{N} \langle G_x(x_k,y_k)-\nabla_x f(x_k,y_k),\; \hat{x}_k -x_k \rangle\nonumber\\
\leq & \frac{2\alpha L}{N} (f(\hat{x}_k,y_k)-f(x_k,y_k))+\frac{3\alpha L^2}{2N}\|\hat{x}_k-x_k\|^2+\frac{\alpha  b L^2r_k^2}{2},
\end{align}
where in the first inequality we used the smoothness of $f(x,y)$, and in the second inequality we used the AM-GM inequality and Lemma \ref{lemma:estimation_error}.
Using the relation $\Phi(\hat{x}_k)\geq f(\hat{x}_k,y_k)$ and the definition of $\hat{x}_k$, we have
\begin{align*}
f(\hat{x}_k,y_k)-f(x_k,y_k)\leq \Phi(\hat{x}_k)-f(x_k,y_k)\leq \Delta_k-L\|\hat{x}_k-x_k\|^2,
\end{align*}
where we applied the relation $\Phi(\hat{x}_k)\leq \Phi(x_k)-L\|x_k-\hat{x}_k\|^2$ in the last step.  Substituting the above inequality into (\ref{eq:lemma3_3}), we further use the relation $$\|\hat{x}_k-x_k\|= \frac{1}{2L}\|\nabla\Phi_{1/2L}(x_k)\|,$$ which is derived from \cite{davis2019stochastic}, to get
\begin{align}\label{eq:lemma3_4}
&\mathbb{E}\left[\left. 2\alpha L\langle G_x^{\mathcal{I}_k}(x_k,y_k),\; \hat{x}_k - x_k \rangle \right| \mathcal{F}_k\right]\nonumber\\
\leq & \frac{2\alpha L\Delta_k}{N}-\frac{\alpha}{8N}\|\nabla\Phi_{1/2L}(x_k)\|^2+\frac{\alpha b L^2 r_k^2}{2}.
\end{align}
For the term $\alpha^2L\left\|G_x^{\mathcal{I}_k}(x_k,y_k)\right\|^2$, we have
\begin{equation}
\begin{aligned}\label{eq:partial_grad_bound}
& \mathbb{E}\left[\left. \alpha^2L\left\|G_x^{\mathcal{I}_k}(x_k,y_k)\right\|^2\right| \mathcal{F}_k \right]\\
\leq\; & \frac{\alpha^2 L}{N} \left\|G_x(x_k,y_k) -\nabla_x f(x_k,y_k) +\nabla_x f(x_k,y_k) \right\|^2 \\
\leq\; & \frac{2\alpha^2 L\Lambda^2}{N}+\frac{\alpha^2 bL^3r_k^2}{2},
\end{aligned}    
\end{equation}
where we applied the Lipschitz continuity and Lemma \ref{lemma:estimation_error} in the last step.

Taking the expectation of (\ref{eq:lemma3_2}) conditioned on $\mathcal{F}_k$ and combining it with (\ref{eq:lemma3_4}) and (\ref{eq:partial_grad_bound}) can derive the final result in Lemma \ref{lemma:lemma1}.

\subsection{Proof of Lemma \ref{lemma:lemma3}}\label{app:proof_of_lemma3}
We divide $\{\Delta_k\}_{k=0}^{K-1}$ into $K/B$ blocks: $\{\Delta_k\}_{k=0}^{B-1},\cdots,\{\Delta_k\}_{jB}^{(j+1)B-1},\cdots,\{\Delta_k\}_{K-B}^{K-1}$, with each block containing $B$ terms. Then, we have
\begin{align}\label{eq:delta_division}
\frac{1}{K}\sum_{k=0}^{K-1}\Delta_k=\frac{B}{K}\sum_{j=0}^{K/B-1}\left(\frac{1}{B}\sum_{k=jB}^{(j+1)B-1}\Delta_k\right).
\end{align}
We provide the following lemma to bound $\Delta_k$, whose proof is provided in \ref{app:lemma2_proof}.
\begin{lemma}\label{lemma:lemma2}
Denote $y^*(x)$ as an arbitrary element in the set $\mathcal{Y}^*(x)=\arg\max_{y\in\mathcal{Y}}f(x,y)$ for any $x\in\mathbb{R}^{d_x}$. Then, for the sequence $\{(x_k,y_k)\}$ derived from ZOB-GDA, we have for any $s\leq k$:
\begin{align*}
 \Delta_k \leq \;&  \alpha\Lambda_0^2(2k-2s+1)+ f(x_{k+1},y_{k+1})-f(x_k,y_k)\\
 & +\frac{1}{2\beta}\left( \|y_k-y^*(x_s)\|^2-\|y_{k+1}-y^*(x_s)\|^2\right).
\end{align*}
\end{lemma}

For the $j$th block,  using the result in Lemma \ref{lemma:lemma2} and letting $s=jB$, we have
\begin{align*}
\sum_{k=jB}^{(j+1)B-1}\Delta_k\leq\;& \alpha\Lambda_0^2 B^2+\frac{R_y^2}{2\beta} +\mathbb{E}\left[f(x_{jB+B},y_{jB+B})-f(x_{jB},y_{jB})\right].
\end{align*}
Substituting the above inequality with $j=0,1,\cdots,K/B-1$ into (\ref{eq:delta_division}), we have
\begin{align*}
\frac{1}{K}\sum_{k=0}^{K-1}\Delta_k\leq\; &\alpha\Lambda_0^2B+\frac{R_y^2}{2\beta B}+\frac{1}{K}\mathbb{E}\left[f(x_{K},y_{K})-f(x_0,y_0)\right].
\end{align*}
We further have
\begin{align*}
& \mathbb{E}\left[f(x_{K},y_{K})-f(x_0,y_0)\right]\\
=\;&\mathbb{E}\left[f(x_{K},y_{K})-f(x_0,y_K)+f(x_0,y_K)-f(x_0,y_0)\right]\\
\leq\; & \alpha \Lambda_0^2K+\Delta_0.
\end{align*}
Combining the above two inequalities leads to the result in Lemma \ref{lemma:lemma3}.

\subsection{Proof of Lemma \ref{lemma:lemma2}}\label{app:lemma2_proof}

Based on the definition of projection, we have for any $y\in\mathcal{Y}$
\begin{align*}
\langle y_{k+1}-y_k-\beta \nabla_y f(x_k,y_k),\; y-y_{k+1}\rangle \geq 0.
\end{align*}
Rearranging this inequality, we can have
\begin{equation}
\begin{aligned}\label{eq:delta_bound_mid1}
&\frac{1}{2\beta}\left(\|y-y_k\|^2-\|y-y_{k+1}\|^2-\|y_{k+1}-y_k\|^2\right)\\
\geq\;&  \langle y-y_{k+1},\; \nabla_y f(x_k,y_k)\rangle\\
=\;&  \langle y-y_{k},\; \nabla_y f(x_k,y_k)\rangle+ \langle y_k-y_{k+1},\; \nabla_y f(x_k,y_k)\rangle.
\end{aligned}
\end{equation}
Using the concavity and smoothness of $f(x,y)$ in $y$, we have
\begin{align*}
\langle y-y_{k},\; \nabla_y f(x_k,y_k)\rangle\geq &f(x_k,y)-f(x_k,y_k),\\
f(x_k,y_{k+1})-f(x_k,y_k)\geq & \langle y_{k+1}-y_{k},\; \nabla_y f(x_k,y_k)\rangle-\frac{L}{2}\|y_{k+1}-y_k\|^2.
\end{align*}
Substituting the above two bounds into (\ref{eq:delta_bound_mid1}) and using the condition $\beta\leq \frac{1}{L}$, we have
\begin{equation}
\begin{aligned}\label{eq:projection_result}
f(x_k,y_{k+1})-f(x_k,y)+\frac{1}{2\beta}\left(\|y-y_{k}\|^2-\|y-y_{k+1}\|^2\right)\geq 0.
\end{aligned}
\end{equation}

Combining the definition of $\Delta_k$ with the above inequality with $y=y^*(x_s)$, we have
\begin{equation}
\begin{aligned}\label{eq:delta_bound}
&\Delta_k=f(x_k,y^*(x_k))-f(x_k,y_k)\\
\leq\;& f(x_k,y^*(x_k))-f(x_k,y_k)+f(x_k,y_{k+1})-f(x_k,y^*(x_s))\\
& +\frac{1}{2\beta}\left(\|y^*(x_s)-y_{k}\|^2-\|y^*(x_s)-y_{k+1}\|^2\right)\\
=\;& \underbrace{f(x_k,y^*(x_k))-f(x_s,y^*(x_s))}_{E_1}+\underbrace{f(x_s,y^*(x_s))-f(x_k,y^*(x_s))}_{E_2}\\
&+\underbrace{f(x_k,y_{k+1})-f(x_{k+1},y_{k+1})}_{E_3}+\left(f(x_{k+1},y_{k+1})-f(x_k,y_k)\right)\\
&+\frac{1}{2\beta}\left(\|y^*(x_s)-y_{k}\|^2-\|y^*(x_s)-y_{k+1}\|^2\right).
\end{aligned}
\end{equation}
Due to that $f(x_s,y^*(x_k))\leq f(x_s,y^*(x_s))$, we have
\begin{align*}
E_1\leq\; & f(x_k,y^*(x_k))-f(x_s,y^*(x_k))\\
\leq\; & \Lambda\|x_k-x_s\|\\
\leq\; & \alpha \Lambda_0^2(k-s).
\end{align*}
Similarly, we also have 
$$E_2\leq  \Lambda \|x_k-x_s\|\leq \alpha \Lambda_0^2(k-s),$$
and
$$E_3 \leq \Lambda \|x_{k+1}-x_k\|\leq \alpha \Lambda_0^2.$$
Substituting these bounds into (\ref{eq:delta_bound}) leads to the final result.

\section{Proof of Theorem \ref{thm:zobsgda}}\label{app:proof_of_thm_ZOBSGDA}
We define some auxiliary notation: $$d(y,z)=\min_{x\in\mathbb{R}^{d_x}} H(x,y;z),\;\;m(z)=\min_{x\in\mathbb{R}^{d_x}}\max_{y\in\mathcal{Y}} H(x,y;z),$$
$$\psi(x,z)=\max_{y\in\mathcal{Y}} H(x,y;z),\;\;x(y,z)=\arg\min_{x\in\mathbb{R}^{d_x}}H(x,y;z),$$
$$x^*(z)=\arg\min_{x\in\mathbb{R}^{d_x}}\psi(x,z),\;\; \mathcal{Y}(z)=\arg\max_{y\in\mathcal{Y}}d(y,z),$$
$$y_+(z_k)=\mathcal{P}_\mathcal{Y}[y_k+\beta\nabla_y H(x(y_k,z_k),y_k,z_k)].$$
Note that $\mathcal{Y}(z)$ is a set, and we use $y(z)$ to denote an arbitrary element in $\mathcal{Y}(z)$.
Recall that we assume $f(x,y)$ is $L$-smooth in $x$ and $y$. Then, if $p> L$, $H(x,y;z)$ is $(p-L)$-strongly convex in $x$ and smooth in $x$ with a constant $(L+p)$. We define the potential function:
\begin{align*}
\phi(x,y,z)=H(x,y;z)-2d(y,z)+2m(z).
\end{align*}
For simplicity, we denote $\phi_k=\phi(x_k,y_k;z_k)$.

Before providing our formal proof, we present some supporting lemmas.

\subsection{Supporting Lemmas for Theorem \ref{thm:zobsgda}}
\begin{lemma}\label{lemma:phi_lower_bound}
For any $x,z\in \mathbb{R}^{d_x}$ and $y\in\mathcal{Y}$, $\phi (x,y;z)$ is lower bounded by $ \underline{f}$.
\end{lemma}
\begin{proof}[Proof of Lemma \ref{lemma:phi_lower_bound}]
We have
\begin{align*}
\phi (x,y;z)=m(z)+ (H(x,y;z)-d(y,z))+(m(z)-d(y,z))\geq m(z)\geq \underline{f},
\end{align*}
where the first inequality follows from the definition of $d(y,z)$ and $m(z)$. The second one holds because $\Phi (x)=\max_{y\in\mathcal{Y}}f(x,y)$ is lower bounded by $\underline{f}$.
    
\end{proof}

\begin{lemma}\label{lemma:prop_collection}
There exists some constants $\sigma_1,\sigma_2$ satisfying
\begin{align*}
&\|x(y,z)-x (y,z')\|\leq \sigma_1\|z-z'\|,\\
&\|x^*(z)-x^*(z')\| \leq \sigma_1\|z-z'\|,\\
&\| x(y,z)-x(y',z)\| \leq \sigma_2 \|y-y'\|,
\end{align*}
for any $y,y'\in\mathcal{Y}$ and $z,z'\in\mathbb{R}^{d_x}$, where $\sigma_1=\frac{p}{p-L}, \sigma_2=\frac{2(p+L)}{p-L}$.
\end{lemma}
Lemma \ref{lemma:prop_collection} follows from the results in \cite[Lemma B.2]{zhang2020single}. Therefore, we omit its proof here.

\begin{lemma}\label{lemma:Lip_of_d}
The dual function $d(y,z)$ is differentiable in $y$ and $L_d$-smooth in $y$, i.e., $\|\nabla_y d(y,z)-\nabla_y d(y',z)\|\leq L_d \|y-y'\|, \forall y,y'\in\mathcal{Y}$, where $L_d=L+L\sigma_2$.
\end{lemma}
\begin{proof}[Proof of Lemma \ref{lemma:Lip_of_d}]
Based on  Danskin's Theorem, we have $\nabla_y d(y,z)=\nabla_y H(x(y,z),y;z)=\nabla_y f(x(y,z),y)$. Then, we have for any $y,y'\in\mathcal{Y}$
\begin{align*}
&\|\nabla_y d(y,z)-\nabla_y d(y',z)\|\\
=\;&\|\nabla_y H(x(y,z),y;z)-\nabla_y H(x(y',z),y';z)\|\\
\leq\; & \|\nabla_y H(x(y,z),y;z)-\nabla_y H(x(y,z),y';z)\|\\
& +\|\nabla_y H(x(y,z),y';z)-\nabla_y H(x(y',z),y';z)\|\\
\leq\; & L\|y-y'\|+ L\|x(y,z)-x(y',z)\|\\
\leq\; &(L+\sigma_2 L)\|y-y'\|,
\end{align*}
where the third step follows from the $L$-smoothness of $H(x,y;z)$ in $y$, and the last step follows from Lemma \ref{lemma:prop_collection}.
\end{proof}

\begin{lemma}\label{lemma:x_prop}
For the sequence $\{(x_k,y_k,z_k)\}$ derived from ZOB-SGDA, we have
\begin{align}\label{eq:prop4}
\mathbb{E}\left[\| x_{k+1}-x(y_k,z_k)\|^2\right]\leq 4\sigma_3^2\mathbb{E}\left[\|x_{k+1}-x_k\|^2\right]+\frac{L^2r_k^2d_x}{(p-L)^2},
\end{align}
where $\sigma_3=\frac{\sqrt{N}+\alpha(p-L)}{\alpha(p-L)}$.
\end{lemma}

\begin{proof}[Proof of Lemma \ref{lemma:x_prop}]
Denote $\mathcal{F}_k=\sigma(x_0, y_0, \mathcal{I}_0, \cdots, i_{k-1},x_k,y_k)$ as a filtration. By Lemma 3.10 in \cite{zhang2020proximal}, we have
\begin{equation}
\begin{aligned}\label{eq:x_prop_mid}
\|x_k-x(y_k,z_k)\|\leq\; & \frac{1}{\alpha(p-L)}\left\|x_k-\mathcal{P}_{\mathcal{X}}[x_k-\alpha\nabla_x H(x_k,y_k;z_k)]\right\|\\
=\; &\frac{1}{\alpha(p-L)}\left\|\alpha\nabla_x H(x_k,y_k;z_k)\right\|,
\end{aligned}
\end{equation}
where $\mathcal{X}=\mathbb{R}^{d_x}$ in our algorithm.
Then, we can get
\begin{align*}
&\|x_{k+1}-x(y_k,z_k)\|^2 \\
\leq\ & 2\|x_{k+1}-x_k\|^2+2\|x_k-x(y_k,z_k)\|^2\\
\leq\;& 2\|x_{k+1}-x_k\|^2+ \frac{2}{\alpha^2(p-L)^2}\left\|\alpha \nabla_x H(x_k,y_k;z_k)\right\|^2\\
\leq\ & 2\|x_{k+1}-x_k\|^2+ \frac{4}{\alpha^2(p-L)^2}\left\|\alpha G_x(x_k,y_k;z_k)\right\|^2\\
& +\frac{4}{(p-L)^2}\|\nabla_xH(x_k,y_k,z_k)-G_x(x_k,y_k;z_k)\|^2\\
\leq\; & 2\|x_{k+1}-x_k\|^2+ \frac{4N}{\alpha^2(p-L)^2}\mathbb{E}\left[\left.\left\|x_k-x_{k+1}\right\|^2\right| \mathcal{F}_k\right] + \frac{L^2r_k^2d_x}{(p-L)^2},
\end{align*}
where the first and third steps follow from the Cauchy-Schwarz inequality. The second step follows from Eq.(\ref{eq:x_prop_mid}). In the last step, we applied Lemma \ref{lemma:estimation_error} and $\mathbb{E}\left[\|\alpha G_x(x_k,y_k;z_k)\|^2\right]=N\mathbb{E}\left[\|x_{k+1}-x_k\|^2\right]$.
Taking the expectation of both sides of the above inequality leads to
\begin{align*}
\mathbb{E}\left[\|x_{k+1}-x(y_k,z_k)\|^2\right] \leq\; & \left(2+\frac{4N}{\alpha^2(p-L)^2}\right) \mathbb{E}\left[\|x_{k+1}-x_k\|^2\right]+\frac{L^2r_k^2d_x}{(p-L)^2}\\
\leq\;&4\sigma_3^2\mathbb{E}\left[\|x_{k+1}-x_k\|^2\right]+\frac{L^2r_k^2d_x}{(p-L)^2}.
\end{align*}

\end{proof}

\begin{lemma}\label{lemma:y_diff_to_x_diff}
For any $k\geq 0$, we have 
$$\mathbb{E}\left[\|y_{k+1}-y_+(z_k)\|^2\right]\leq \kappa \mathbb{E}\left[\|x_{k+1}-x_k\|^2\right]+\frac{\beta^2 L^2 r_k^2 d_x}{2},$$
where $\kappa=(8\sigma_3^2+2)\beta^2L^2$.
\end{lemma}
\begin{proof}[Proof of Lemma \ref{lemma:y_diff_to_x_diff}]
By the non-expansiveness of the projection operator, we have
\begin{align*}
&\|y_{k+1}-y_+(z_k)\|^2\\
=\;& \left\|\mathcal{P}_\mathcal{Y}[y_k+\beta\cdot \nabla_y H(x_k,y_k;z_k)]-\mathcal{P}_\mathcal{Y}[y_k+\beta\cdot \nabla_y H(x(y_k,z_k),y_k;z_k)]\right\|^2\\
\leq\; &\beta^2\|\nabla_y H(x(y_k,z_k),y_k;z_k)-\nabla_y H(x_{k},y_k;z_k)\|^2\\
\leq\; & \beta^2 L^2\|x_{k}-x(y_k,z_k)\|^2.
\end{align*}
where in the first step we used the non-expansiveness of projection operations and in the third step we used the $L$-smoothness of $H(x,y;z)$ in $x$. Then, taking the expectation of the above inequality leads to
\begin{align*}
&\mathbb{E}\left[\|y_{k+1}-y_+(z_k)\|^2\right]\\
\leq\; & \beta^2 L^2\mathbb{E}\left[\|x_{k}-x(y_k,z_k)\|^2\right]\\
\leq\;& 2\beta^2 L^2\mathbb{E}\left[\|x_{k+1}-x_k\|^2\right]+2\beta^2 L^2\mathbb{E}\left[\|x_{k+1}-x(y_k,z_k)\|^2\right]\\
\leq\;& \left(8\beta^2 L^2\sigma_3^2+2\beta^2L^2\right)\mathbb{E}\left[\|x_{k+1}-x_k\|^2\right]+\frac{2\beta^2 L^4 r_k^2d_x}{(p-L)^2}\\
\leq\;& \left(8\beta^2 L^2\sigma_3^2+2\beta^2L^2\right)\mathbb{E}\left[\|x_{k+1}-x_k\|^2\right]+\frac{\beta^2 L^2r_k^2 d_x}{2},
\end{align*}
where in the second step we applied Lemma \ref{lemma:x_prop} and in the last step we applied the inequality (\ref{eq:prop4}) and the condition $p\geq 3L$.
\end{proof}

\begin{lemma}\label{lemma:opti_x_to_y}
For any $k\geq 0$, we have
$$\beta (p-L)\|x^*(z_k)-x(y_+(z_k),z_k)\|^2\leq (1+\beta L +\beta L \sigma_2)R_y\|y_k-y_+(z_k)\|.$$
\end{lemma}
Lemma \ref{lemma:opti_x_to_y} comes from \cite[Lemma B.10]{zhang2020single}.

\subsection{Formal Proof of Theorem \ref{thm:zobsgda}}
In this subsection, we provide the formal proof of Theorem \ref{thm:zobsgda}. The proof mainly contains three steps as follows.

\textbf{Step 1: Derive a Bound for the Stationarity Measure.} For ZOB-SGDA, we analyze the convergence of $\|\mathfrak{g}(x,y)\|$. Recall the definition of $\mathfrak{g}(x,y)$:
\begin{align*}
\mathfrak{g}(x,y)=
\begin{pmatrix}
    \mathfrak{g}_x(x,y) \\
    \mathfrak{g}_y(x,y)
\end{pmatrix}=
\begin{pmatrix}
    \nabla_x f(x,y) \\
    \frac{1}{\beta}\left(y-\mathcal{P}_\mathcal{Y} \left[y+\beta \nabla_y f(x,y)\right]\right)
\end{pmatrix},
\end{align*}
for any $(x,y)\in \mathbb{R}^{d_x}\times\mathcal{Y}$. Then, we provide a bound on the stationarity measure in the following lemma.
\begin{lemma}\label{lemma:stationarity_bound}
For any $\{x_k,y_k,z_k\}_{k\geq 0}$ derived from ZOB-SGDA, we have
\begin{align*}
\mathbb{E}\left[\|\mathfrak{g}(x_k,y_k)\|^2\right]\leq\;& \mathbb{E}\left[\left(\frac{3N}{\alpha^2}+8L^2\sigma_3^2+6p^2\right)\|x_{k+1}-x_k\|^2 \right]+\frac{5L^2r_k^2d_x}{4}\\
&  + \mathbb{E}\left[\frac{2}{\beta^2}\|y_k-y_+ (z_k)\|^2 + 6p^2\|x_{k+1}-z_k\|^2\right].
\end{align*}
\end{lemma}
\begin{proof}[Proof of Lemma \ref{lemma:stationarity_bound}]
Based on the update of $x:\; x_{k+1}=x_k-\alpha\cdot G_x^{\mathcal{I}_k}(x_k,y_k;z_k)$, we have
\begin{align*}
&\mathbb{E}\left[\|\mathfrak{g}_x(x_k,y_k)\|^2\right]\\
=\; & \mathbb{E}\left[\|G_x(x_k,y_k;z_k)+\nabla_x H(x_k,y_k;z_k)-G_x(x_k,y_k;z_k)-p(x_k-z_k)\|^2\right]\\
\leq\;&\mathbb{E}\left[3\|G_x(x_k,y_k;z_k)\|^2+3\|G_x(x_k,y_k;z_k)-\nabla_x H(x_k,y_k;z_k)\|^2+3p^2\|x_k-z_k\|^2\right]\\
\leq\; & \mathbb{E}\left[\frac{3N}{\alpha^2}\|x_{k+1}-x_k\|^2+3p^2\|x_k-z_k\|^2\right]+\frac{3L^2r_k^2d_x}{4}\\
\leq\; & \mathbb{E}\left[ \left(\frac{3N}{\alpha^2}+6p^2\right)\|x_{k+1}-x_k\|^2+6p^2\|x_{k+1}-z_k\|^2\right]+\frac{3L^2r_k^2d_x}{4},
\end{align*}
where in the third step we applied the relation
$$\mathbb{E}\left[\frac{3N}{\alpha^2}\|x_{k+1}-x_k\|^2\right]=\mathbb{E}\left[3N\|G_x^{\mathcal{I}_k}(x_k,y_k)\|^2\right]=\mathbb{E}\left[3\|G_x(x_k,y_k)\|^2\right].$$

For the term $\|\mathfrak{g}_y(x_k,y_k)\|^2$, we can get
\begin{align*}
&\mathbb{E}\left[\|\mathfrak{g}_y(x_k,y_k)\|^2\right]\\
=\;& \frac{1}{\beta^2}\mathbb{E}\left[\|y_{k+1}-y_k\|^2\right]\\
\leq\; & \frac{2}{\beta^2}\mathbb{E}\left[\|y_{k+1}-y_+(z_k)\|^2\right]+\frac{2}{\beta^2}\mathbb{E}\left[\|y_+(z_k)-y_k\|^2\right]\\
\leq\; & \frac{8\kappa^2}{\beta^2}\mathbb{E}\left[\|x_{k+1}-x_k\|^2\right]+\frac{2}{\beta^2}\mathbb{E}\left[\|y_+(z_k)-y_k\|^2\right]+\frac{L^2r_k^2d_x}{2}\\
=\; & 8L^2\sigma_3^2\mathbb{E}\left[\|x_{k+1}-x_k\|^2\right]+\frac{2}{\beta^2}\mathbb{E}\left[\|y_+(z_k)-y_k\|^2\right]+\frac{L^2r_k^2d_x}{2},
\end{align*}
where we applied Lemma \ref{lemma:y_diff_to_x_diff} in the third step. Finally, combining it with the bound on $\mathbb{E}[\|\mathfrak{g}_x(x_k,y_k)\|^2]$ leads to the desired result.
\end{proof}

\textbf{Step 2: Derive a bound on the one-step drift of potential function.}   We provide a bound on the one-step drift in the following lemma.
\begin{lemma}\label{lemma:potential_one_step_drift}
Suppose the assumptions and conditions in Theorem \ref{thm:zobsgda} hold. For any $\{x_k,y_k,z_k\}$ derived from ZOB-SGDA, we have
\begin{equation}
\begin{aligned}\label{eq:potential_drift}
& \mathbb{E}\left[\phi_k-\phi_{k+1}\right]\\
\geq\; &  \mathbb{E}\left[\underbrace{\frac{1}{8\alpha}\|x_{k+1}-x_k\|^2+\frac{1}{8\beta}\|y_{k}-y_+(z_k)\|^2+\frac{p}{8\gamma}\|z_{k+1}-z_k\|^2}_{T_1}\right]\\
& - \mathbb{E}\left[\underbrace{24p\gamma\|x^*(z_k)-x(y_+(z_k),z_k)\|^2}_{T_2}\right] -\frac{L^3r_k^2\alpha^2d_x}{N}\\
& -12\beta^2L^2r_k^2\sigma_2^2 p\gamma d_x-\frac{\beta L^2r_k^2d_x}{8} -\frac{L^2r_k^2 d_x}{8N}.
\end{aligned}
\end{equation}
\end{lemma}
\begin{proof}[Proof of Lemma \ref{lemma:potential_one_step_drift}] First, we provide the following lemma to characterize the descent in primal steps.
\begin{lemma}\label{lemma:primal_descent}
For any $k\geq 0$, we have
\begin{align*}
&H(x_k,y_k;z_k)-H(x_{k+1},y_{k+1};z_{k+1})\\
\geq\; & \left(\frac{1}{\alpha}-\frac{p+L+1}{2}\right)\|x_{k+1}-x_k\|^2 -\frac{L}{2}\| y_k-y_{k+1} \|^2+\frac{p}{2\gamma}\|z_{k+1}-z_k\|^2\\
& +\langle \nabla_y H(x_{k+1},y_k;z_k), \; y_k-y_{k+1}\rangle-\frac{L^2r_k^2d_{x}}{8N}.
\end{align*}
\end{lemma}
\begin{proof}[Proof of Lemma \ref{lemma:primal_descent}]
By the update of $x$, we use the smoothness of $H$ to get
\begin{equation}
\begin{aligned}\label{eq:primal_descent_r_mid1}
& H(x_{k+1},y_k;z_k)-H(x_k,y_k;z_k)\\
\leq\;& \langle \nabla_x H(x_k,y_k;z_k),\; x_{k+1}-x_k \rangle +\frac{p+L}{2}\|x_{k+1}-x_k\|^2\\
=\; &  \langle G_x(x_k,y_k;z_k),\; x_{k+1}-x_k \rangle +\frac{p+L}{2}\|x_{k+1}-x_k\|^2\\
& +\langle \nabla_x H(x_k,y_k;z_k)-G_x(x_k,y_k;z_k),\; x_{k+1}-x_k \rangle\\
\leq\; & \langle G_x(x_k,y_k;z_k),\; x_{k+1}-x_k \rangle +\frac{p+L+1}{2}\|x_{k+1}-x_k\|^2+\frac{L^2r_k^2d_{x}}{8N}\\
\leq\; & \left(-\frac{1}{\alpha}+\frac{p+L+1}{2}\right)\|x_{k+1}-x_k\|^2+\frac{L^2r_k^2d_{x}}{8N},
\end{aligned}    
\end{equation}
where in the second inequality we applied AM-GM inequality and the fact that only the entries of $\mathcal{I}_k$ in $x_{k+1}-x_k$ are nonzero. Similarly, we can use the $L$-smoothness of $H(x,y;z)$ to get
\begin{align}\label{eq:primal_descent_r_mid2}
& H(x_{k+1},y_k;z_k)-H(x_{k+1},y_{k+1};z_k)\notag\\
\geq\;& \langle \nabla_y H(x_{k+1},y_k;z_k), \; y_k-y_{k+1}\rangle -\frac{L}{2}\|y_{k+1}-y_k\|^2.
\end{align}
Based on the update of $z$, we have
\begin{align}\label{eq:primal_descent_r_mid3}
H(x_{k+1},y_{k+1};z_k)-H(x_{k+1},y_{k+1};z_{k+1})\geq \frac{p}{2\gamma}\|z_{k+1}-z_k\|^2.
\end{align}
Combining the results in (\ref{eq:primal_descent_r_mid1}), (\ref{eq:primal_descent_r_mid2}), and (\ref{eq:primal_descent_r_mid3}) leads to the final result.
\end{proof}

We further provide the following two lemmas to characterize the one-step drift of $d(y,z)$ and $m(z)$ in ZOB-SGDA. Their proofs follow from \cite[Lemma B.6 and Lemma B.7]{zhang2020single}.
\begin{lemma}\label{lemma:dual_ascent}
For any $k$, we have
\begin{align*}
&d(y_{k+1},z_{k+1})-d(y_k,z_k)\\
\geq\; & \langle \nabla_y H(x(y_k,z_k),y_k;z_k),\; y_{k+1}-y_k\rangle-\frac{L_d}{2}\|y_{k+1}-y_k\|^2\\
&+\frac{p}{2}\langle z_{k+1}-z_k,\;z_{k+1}+z_k-2x(y_{k+1},z_{k+1})\rangle.
\end{align*}
\end{lemma}
\begin{proof}[Proof of Lemma \ref{lemma:dual_ascent}]
Using the smoothness of $d(y,z)$ in $y$ provided in Lemma \ref{lemma:Lip_of_d}, we have
\begin{align*}
&d(y_{k+1},z_k)-d(y_k,z_k)\\
\geq\;& \langle \nabla_y d(y_k,z_k),y_{k+1}-y_k\rangle-\frac{L_d}{2}\|y_{k+1}-y_k\|^2\\
=\;& \langle \nabla_y H(x(y_k,z_k),y_k;z_k),y_{k+1}-y_k\rangle-\frac{L_d}{2}\|y_{k+1}-y_k\|^2,
\end{align*}
where in the second step we used $ \nabla_y d(y_k,z_k)=\nabla_y H(x(y_k,z_k),y_k;z_k)$. Also, we have
\begin{align*}
& d(y_{k+1},z_{k+1})-d(y_{k+1},z_k)\\
=\; & H(x(y_{k+1},z_{k+1}),y_{k+1};z_{k+1})-H(x(y_{k+1},z_{k}),y_{k+1};z_{k})\\
\geq\; & H(x(y_{k+1},z_{k+1}),y_{k+1};z_{k+1})-H(x(y_{k+1},z_{k+1}),y_{k+1};z_{k})\\
=\;& \frac{p}{2}\|x_{k+1}(y_{k+1},z_{k+1})-z_{k+1}\|^2-\frac{p}{2}\|x_{k+1}(y_{k+1},z_{k+1})-z_k\|^2\\
=\;& \frac{p}{2}\langle z_{k+1}-z_k,\; z_{k+1}+z_k- 2x(y_{k+1},z_{k+1})\rangle.
\end{align*}
Finally, combining the above two inequalities leads to the desired result.

\end{proof}

\begin{lemma}\label{lemma:proximal_descent}
For any $k$, we have
\begin{align*}
m(z_{k+1})-m(z_k)\leq \frac{p}{2}\langle z_{k+1}-z_k, \; z_{k+1}+z_k-2x(y(z_{k+1}),z_k)\rangle,
\end{align*}
where $y(z_{k+1})$ is an arbitrary element in $\mathcal{Y}(z_{k+1})$.
\end{lemma}
\begin{proof}
For any $y(z)\in\mathcal{Y}(z)$, we have  
$$m(z)=\max_{y\in \mathcal{Y}}d(y,z)= d(y(z),z).$$
Therefore, we have
\begin{align*}
&m(z_{k+1})-m(z_k)\\
\leq\; & d(y(z_{k+1}),z_{k+1})- d(y(z_{k+1}),z_k)\\
=\;& H(x(y(z_{k+1}),z_{k+1}),y(z_{k+1});z_{k+1})-H(x(y(z_{k+1}),z_{k}),y(z_{k+1});z_{k})\\
\leq\;& H(x(y(z_{k+1}),z_{k}),y(z_{k+1});z_{k+1})-H(x(y(z_{k+1}),z_{k}),y(z_{k+1});z_{k})\\
=\;& \frac{p}{2}\langle z_{k+1}-z_k,\; z_{k+1}+z_k-2x(y(z_{k+1}),z_k) \rangle,
\end{align*}
where in the first step and third step we used the definitions of $y(z)$ and $x(y,z)$, respectively.
\end{proof}

Now we can bound the one-step drift of the potential function. Using the results in Lemmas \ref{lemma:primal_descent}, \ref{lemma:dual_ascent}, and \ref{lemma:proximal_descent}, we have
\begin{align*}
&\phi_k-\phi_{k+1}\\
\geq\; & \left(\frac{1}{\alpha}-\frac{p+L+1}{2}\right)\|x_{k+1}-x_k\|^2-\frac{L+2L_d}{2}\|y_{k+1}-y_k\|^2\\
&+\frac{p}{2\gamma}\|z_{k+1}-z_k\|^2 +\langle \nabla_y H(x_{k+1},y_k;z_k), y_{k+1}-y_k\rangle\\
& +2\langle \nabla_yH(x(y_k,z_k),y_k;z_k)-\nabla_y H(x_{k+1},y_k;z_k),\; y_{k+1}-y_k\rangle\\
& + 2p(z_{k+1}-z_k)^T\left( x(y(z_{k+1}),z_k)-x(y_{k+1},z_{k+1}) \right)-\frac{L^2r_k^2d_{x}}{8N}.
\end{align*}
Using the property of the projection operator:
$$\langle y_{k+1}-(y_k+\beta \nabla_y H(x_{k},y_k;z_k)),\; y_{k+1}-y \rangle\leq 0,$$ 
for any $y\in\mathcal{Y}$, and setting $y=y_k$, we have 
\begin{align*}
\langle \nabla_y H(x_{k},y_k;z_k), y_{k+1}-y_k\rangle\geq \frac{1}{\beta}\|y_{k+1}-y_k\|^2.
\end{align*}
Therefore,
\begin{align*}
&\langle \nabla_y H(x_{k+1},y_k;z_k), y_{k+1}-y_k\rangle\\
=\;& \langle \nabla_y H(x_{k},y_k;z_k), y_{k+1}-y_k\rangle+\langle \nabla_y H(x_{k+1},y_k;z_k)-\nabla_y H(x_{k},y_k;z_k), y_{k+1}-y_k\rangle\\
\geq\;&\frac{1}{\beta}\|y_{k+1}-y_k\|^2-\frac{L}{2}\|y_{k+1}-y_k\|^2-\frac{L}{2}\|x_{k+1}-x_k\|^2\\
=\;& \left(\frac{1}{\beta}-\frac{L}{2}\right)\|y_{k+1}-y_k\|^2-\frac{L}{2}\|x_{k+1}-x_k\|^2,
\end{align*}
where the first inequality follows from AM-GM inequality and the smoothness of $H$. Then, we can further get
\begin{equation}
\begin{aligned}\label{eq:potential_bound1}
&\phi_k-\phi_{k+1}\\
\geq\; & \left(\frac{1}{\alpha}-\frac{p+2L+1}{2}\right)\|x_{k+1}-x_k\|^2+\left( \frac{1}{\beta}-(L+L_d) \right)\|y_{k+1}-y_k\|^2+\frac{p}{2\gamma}\|z_{k+1}-z_k\|^2 \\
& +2\langle \nabla_yH(x(y_k,z_k),y_k;z_k)-\nabla_y H(x_{k+1},y_k;z_k),\; y_{k+1}-y_k\rangle\\
& + 2p(z_{k+1}-z_k)^T\left( x(y(z_{k+1}),z_k)-x(y_{k+1},z_{k+1}) \right)-\frac{L^2r_k^2d_{x}}{8N}.
\end{aligned}    
\end{equation}

Besides, we have
\begin{equation}
\begin{aligned}\label{eq:bound_term1}
&2p(z_{k+1}-z_k)^T\left( x(y(z_{k+1}),z_k)-x(y_{k+1},z_{k+1}) \right)\\
=\; & 2p(z_{k+1}-z_k)^T\left( x(y(z_{k+1}),z_k)-x(y(z_{k+1}),z_{k+1})\right)\\
& +2p(z_{k+1}-z_k)^T\left( x(y(z_{k+1}),z_{k+1})-x(y_{k+1},z_{k+1})  \right)\\
\geq\; & -2p\sigma_1 \|z_{k+1}-z_k\|^2+2p(z_{k+1}-z_k)^T\left( x(y(z_{k+1}),z_{k+1})-x(y_{k+1},z_{k+1})  \right)\\
\geq\; &  -2p\sigma_1\|z_{k+1}-z_k\|^2-\frac{p}{6\gamma}\|z_{k+1}-z_k\|^2-6p\gamma\|x(y(z_{k+1}),z_{k+1})-x(y_{k+1},z_{k+1})\|^2,
\end{aligned}
\end{equation}
where the second step follows from Lemma \ref{lemma:prop_collection} and the third step follows from the AM-GM inequality. We also have the bound:
\begin{equation}
\begin{aligned}\label{eq:bound_term2}
&\mathbb{E}\left[2\langle \nabla_yH(x(y_k,z_k),y_k;z_k)-\nabla_y H(x_{k+1},y_k;z_k),\; y_{k+1}-y_k\rangle\right]\\
\geq\ & \mathbb{E}\left[-2L\|x_{k+1}-x(y_k,z_k)\|\cdot \|y_{k+1}-y_k\|\right]\\
\geq\; & \mathbb{E}\left[-L\sigma_3^2\|y_{k+1}-y_k\|^2-L\sigma_3^{-2}\|x_{k+1}-x(y_k,z_k)\|^2\right]\\
\geq\; & \mathbb{E}\left[-L\sigma_3^2\|y_{k+1}-y_k\|^2-4L\|x_{k+1}-x_k\|^2\right]-\frac{L^3r_k^2\alpha^2d_x}{\left(\sqrt{N}+\alpha (p-L)\right)^2}\\
\geq\; & \mathbb{E}\left[-L\sigma_3^2\|y_{k+1}-y_k\|^2-4L\|x_{k+1}-x_k\|^2\right]-\frac{L^3r_k^2\alpha^2d_x}{N},
\end{aligned}    
\end{equation}
where the first step follows from the smoothness of $H(x,y;z)$ and the third step follows from Lemma \ref{lemma:x_prop}. In the last step, we used the fact that $p>L$. Taking the expectation on both sides of (\ref{eq:potential_bound1}) and combining it with the bounds in (\ref{eq:bound_term1}) and (\ref{eq:bound_term2}), we can get
\begin{equation}
\begin{aligned}\label{eq:potential_bound2}
&\mathbb{E}\left[\phi_k-\phi_{k+1}\right]\\
\geq\; & \mathbb{E}\left[\left(\frac{1}{\alpha}-\frac{p+2L+1}{2}-4L\right)\|x_{k+1}-x_k\|^2+\left(\frac{1}{\beta}-L-L_d-L\sigma_3^2 \right)\|y_{k+1}-y_k\|^2\right]\\
& +\mathbb{E}\left[\left(\frac{p}{2\gamma}-2p\sigma_1-\frac{p}{6\gamma}\right)\|z_{k+1}-z_k\|^2\right]\\
& -\mathbb{E}\left[6p\gamma\|x(y(z_{k+1}),z_{k+1})-x(y_{k+1},z_{k+1})\|^2\right]-\frac{L^3r_k^2\alpha^2d_x}{N}-\frac{L^2r_k^2 d_x}{8N}.
\end{aligned}    
\end{equation}
Using the conditions $\beta\leq \min\left\{ \frac{1}{12L},\frac{\alpha^2(p-L)^2}{4L(\sqrt{N}+\alpha(p-L))^2} \right\}$, we have
$$L+L_d\leq 6L\leq \frac{1}{2\beta},$$
and
$$L\sigma_3^2=\frac{L(\sqrt{N}+\alpha(p-L))^2}{\alpha^2(p-L)^2}\leq \frac{1}{4\beta}.$$
Therefore, we have
\begin{align}\label{eq:beta_bound}
\frac{1}{\beta}-L-L_d-L\sigma_3^2\geq\frac{1}{4\beta}.
\end{align}

Using Lemma \ref{lemma:y_diff_to_x_diff}, we have the bound on $\mathbb{E}[\|y_{k+1}-y_k\|^2]$:
\begin{align*}
&\mathbb{E}\left[\|y_{k+1}-y_k\|^2\right]\\
\geq\ & \mathbb{E}\left[\frac{1}{2}\|y_k-y_+(z_k)\|^2-\|y_{k+1}-y_+(z_k)\|^2\right]\\
\geq\ & \mathbb{E}\left[\frac{1}{2}\|y_k-y_+(z_k)\|^2-\kappa\|x_{k+1}-x_k\|^2\right]-\frac{\beta^2L^2r_k^2d_x}{2}.
\end{align*}
As for the bound on $\mathbb{E}[\|x^*(z_{k+1})-x(y_{k+1},z_{k+1})\|^2]$, we have
\begin{align*}
&\mathbb{E}\left[\|x^*(z_{k+1})-x(y_{k+1},z_{k+1})\|^2\right]\\
\leq\; & \mathbb{E}\left[2\|x^*(z_{k+1})-x^*(z_k)+x^*(z_k)-x(y_+(z_k),z_k)\|^2\right.\\
& \left.+2\|x(y_+(z_k),z_k)-x(y_{k+1},z_k)+x(y_{k+1},z_k)-x(y_{k+1},z_{k+1})\|^2\right]\\
\leq\ & \mathbb{E}\left[4\|x^*(z_{k+1})-x^*(z_k)\|^2+4\|x^*(z_k)-x(y_+(z_k),z_k)\|^2\right]\\
& + \mathbb{E}\left[4\|x(y_+(z_k),z_k)-x(y_{k+1},z_k)\|^2+4\|x(y_{k+1},z_k)-x(y_{k+1},z_{k+1})\|^2\right]\\
\leq\; & \mathbb{E}\left[4\sigma_1^2\|z_{k+1}-z_k\|^2+4\|x^*(z_k)-x(y_+(z_k),z_k)\|^2\right]\\
& +\mathbb{E}\left[4\sigma_2^2\|y_+(z_k)-y_{k+1}\|^2+4\sigma_1^2\|z_k-z_{k+1}\|^2\right]\\
\leq\; & \mathbb{E}\left[8\sigma_1^2\|z_{k+1}-z_k\|^2+4\|x^*(z_k)-x(y_+(z_k),z_k)\|^2\right.\\
&\left.+4\sigma_2^2\kappa\|x_{k+1}-x_k\|^2\right]+2\beta^2L^2r_k^2\sigma_2^2 d_x,
\end{align*}
where the first two steps follow from the Cauchy-Schwarz inequality and the last two steps follow from Lemma \ref{lemma:prop_collection} and \ref{lemma:y_diff_to_x_diff}.
Combining the above two bounds with (\ref{eq:beta_bound}) and (\ref{eq:potential_bound2}) leads to
\begin{align*}
& \mathbb{E}\left[\phi_k-\phi_{k+1}\right]\\
\geq\; &  \mathbb{E}\left[\left(\frac{1}{\alpha}-\frac{p+2L+1}{2}-4L-\frac{\kappa}{4\beta }-24p\gamma\sigma_2^2\kappa \right)\|x_{k+1}-x_k\|^2\right]\\
& +\mathbb{E}\left[\frac{1}{8\beta}\|y_{k}-y_+(z_k)\|^2+\left(\frac{p}{2\gamma}-2p\sigma_1-\frac{p}{6\gamma}-48p\gamma \sigma_1^2\right)\|z_{k+1}-z_k\|^2\right]\\
& -\mathbb{E}\left[24p\gamma\|x^*(z_k)-x(y_+(z_k),z_k)\|^2\right]\\
& -\frac{L^3r_k^2\alpha^2d_x}{N}-12p\gamma d_x\sigma_2^2\beta^2L^2r_k^2-\frac{\beta L^2r_k^2d_x}{8} -\frac{L^2r_k^2 d_x}{8N}.
\end{align*}
Based on the condition $\alpha\leq  \frac{1}{p+10L+1}$, we have $\alpha<\frac{1}{p+10L}\leq \frac{1}{13L}$ and $\frac{p+10L+1}{2}\leq \frac{1}{2\alpha}$. Using the condition $\beta\leq \min\left\{ \frac{1}{12L},\frac{\alpha^2(p-L)^2}{4L(\sqrt{N}+\alpha(p-L))^2} \right\}$, we have
$$\frac{\kappa}{4\beta}=2\beta L^2\sigma_3^2+\frac{1}{2}\beta L^2 < L < \frac{1}{8\alpha}.$$
Furthermore, using the condition $\gamma\leq \frac{1}{768p\beta}$, we can get
$$24p\gamma \sigma_2^2\kappa\leq \frac{\sigma_2^2\kappa}{32\beta}\leq \frac{\kappa}{2\beta}\leq \frac{1}{4\alpha}.$$
Therefore, we have
$$\frac{1}{\alpha}-\frac{p+2L+1}{2}-4L-\frac{\kappa}{4\beta }-24p\gamma\sigma_2^2\kappa > \frac{1}{8\alpha}.$$

Using the condition $\gamma\leq \frac{1}{36}$, we have $2p\sigma_1\leq 3p \leq \frac{p}{12\gamma}$ and $48p
\gamma \sigma_1^2 \leq 3p \leq \frac{p}{12\gamma}$. Thus, we have
$$\frac{p}{2\gamma}-2p\sigma_1-\frac{p}{6\gamma}-48p\gamma \sigma_1^2> \frac{p}{8\gamma}. $$

Then, we can combine the above results to get
\begin{align*}
& \mathbb{E}\left[\phi_k-\phi_{k+1}\right]\\
\geq\; &  \mathbb{E}\left[\frac{1}{8\alpha}\|x_{k+1}-x_k\|^2+\frac{1}{8\beta}\|y_{k}-y_+(z_k)\|^2+\frac{p}{8\gamma}\|z_{k+1}-z_k\|^2\right]\\
& - \mathbb{E}\left[24p\gamma\|x^*(z_k)-x(y_+(z_k),z_k)\|^2\right] -\frac{L^3r_k^2\alpha^2d_x}{N}\\
& -12p\gamma \beta^2L^2r_k^2\sigma_2^2 d_x-\frac{\beta L^2r_k^2d_x}{8} -\frac{L^2r_k^2 d_x}{8N}.
\end{align*}
\end{proof}

\textbf{Step 3: Combine the above two bounds to get the convergence rate.}
Then, we are ready to prove the final result. We consider two situations in (\ref{eq:potential_drift}): (1) $\frac{1}{2}T_1\leq T_2$, (2) $\frac{1}{2}T_1 > T_2$. In the first case, we have
\begin{align*}
&\frac{1}{16\alpha}\|x_{k+1}-x_k\|^2+\frac{1}{16\beta}\|y_{k}-y_+(z_k)\|^2+\frac{p}{16\gamma}\|z_{k+1}-z_k\|^2\\
\leq\; & 24p\gamma\|x^*(z_k)-x(y_+(z_k),z_k)\|^2.
\end{align*}

Let $t_1=384p R_y\frac{1+\beta L+\beta L\sigma_2}{p-L}$. Then, using Lemma \ref{lemma:opti_x_to_y}, we have
\begin{align*}
&\|y_{k}-y_+(z_k)\|^2\\
\leq\;& 384p\gamma\beta \|x^*(z_k)-x(y_+(z_k),z_k)\|^2\\
\leq\;& 384p\gamma\beta R_y\frac{1+\beta L+\beta L\sigma_2}{\beta(p-L)}\|y_{k}-y_+(z_k)\|,
\end{align*}
which leads to $\|y_{k}-y_+(z_k)\|\leq t_1 \gamma$. Therefore, we can get
\begin{align*}
&\|z_{k+1}-z_k\|^2\\
\leq\;&  384\gamma^2\|x^*(z_k)-x(y_+(z_k),z_k)\|^2\\
\leq\;&384\gamma^2 R_y\frac{1+\beta L+\beta L\sigma_2}{\beta(p-L)}\|y_{k}-y_+(z_k)\|\\
=\;& \frac{t_1^2\gamma^3}{p\beta},
\end{align*}
and
\begin{align*}
&\|x_{k+1}-x_k\|^2\\
\leq\;& 384p\alpha\gamma \|x^*(z_k)-x(y_+(z_k),z_k)\|^2\\
\leq\;& 384p\alpha\gamma R_y\frac{1+\beta L+\beta L\sigma_2}{\beta(p-L)}\|y_{k}-y_+(z_k)\|\\
\leq\;& \frac{\alpha t_1^2\gamma^2}{\beta}.
\end{align*}
Combining the above three inequalities with Lemma \ref{lemma:stationarity_bound}, we can get
\begin{align*}
&\frac{1}{K}\sum_{k=0}^{K-1}\mathbb{E}\left[\|\mathfrak{g}(x_k,y_k)\|^2\right]\\
\leq\;& \left(\frac{3N}{\alpha^2}+8L^2\sigma_3^2+6p^2\right)\frac{\alpha t_1^2\gamma^2}{\beta} + \frac{2t_1\gamma^2}{\beta^2}+ \frac{6p^2t_1^2\gamma}{p\beta}+\frac{5L^2d_x\sum_{k=0}^{K-1}r_k^2}{4K}.
\end{align*}
By $\gamma\leq \frac{1}{\sqrt{KN}}$ and $\sum_{k=0}^{K-1} r_k^2\leq \frac{1}{b}$, we can get $\frac{1}{K}\sum_{k=0}^{K-1}\mathbb{E}\left[\|\mathfrak{g}(x_k,y_k)\|^2\right]\leq\mathcal{O}\left(\frac{N}{K}+\sqrt{\frac{N}{K}}\right)$, which leads to
\begin{align*}
\frac{1}{K}\sum_{k=0}^{K-1}\mathbb{E}\left[\|\mathfrak{g}(x_k,y_k)\|\right]\leq\mathcal{O}\left(\sqrt{\frac{N}{K}+\sqrt{\frac{N}{K}}}\right)= \mathcal{O}\left(\left(\frac{N}{K}\right)^{\frac{1}{4}}\right).
\end{align*}

In the second case, we have
\begin{align*}
&\frac{1}{16\alpha}\|x_{k+1}-x_k\|^2+\frac{1}{16\beta}\|y_{k}-y_+(z_k)\|^2+\frac{p}{16\gamma}\|z_{k+1}-z_k\|^2\\
\geq\; & 24p\gamma\|x^*(z_k)-x(y_+(z_k),z_k)\|^2.
\end{align*}
Then, according to (\ref{eq:potential_drift}), we have
\begin{align*}
& \mathbb{E}\left[\phi_k-\phi_{k+1}\right]\\
\geq\; &  \mathbb{E}\left[\frac{1}{16\alpha}\|x_{k+1}-x_k\|^2+\frac{1}{16\beta}\|y_{k}-y_+(z_k)\|^2+\frac{p\gamma}{16}\|x_{k+1}-z_k\|^2\right]\\
& -\frac{L^3r_k^2\alpha^2d_x}{N}-12\beta^2L^2r_k^2\sigma_2^2 p\gamma d_x-\frac{\beta L^2r_k^2d_x}{8} - \frac{L^2r_k^2d_x}{8N}\\
\geq\; & t_2\mathbb{E}\left[\|\mathfrak{g}(x_k,y_k)\|^2\right]-\frac{5t_2L^2r_k^2d_x}{4}-\frac{L^3r_k^2\alpha^2d_x}{N}-12\beta^2L^2r_k^2\sigma_2^2 d_x-\frac{\beta L^2r_k^2d_x}{8} -\frac{L^2r_k^2d_x}{8N},
\end{align*}
where $t_2=\min \left\{ \frac{1}{16\alpha\left(\frac{3N}{\alpha^2}+8L^2\sigma_3^2+6p^2\right)},\frac{\beta}{32}, \frac{\gamma}{96p} \right\}$. Consequently, we can obtain
\begin{align*}
\mathbb{E}\left[\|\mathfrak{g}(x_k,y_k)\|^2\right]\leq \;& \frac{ \mathbb{E}\left[\phi_k-\phi_{k+1}\right]}{t_2}+\frac{5L^2r_k^2d_x}{4}+\frac{L^3r_k^2\alpha^2d_x}{N t_2}\\
& +\frac{12\beta^2L^2r_k^2\sigma_2^2 d_x}{t_2}+\frac{\beta L^2r_k^2d_x}{8t_2} +\frac{L^2r_k^2d_x}{8Nt_2}.
\end{align*}
Taking the summation of the above inequality from $k=0$ to $K-1$ leads to
\begin{align*}
& \frac{1}{K}\sum_{k=0}^{K-1}\mathbb{E}\left[\|\mathfrak{g}(x_k,y_k)\|^2\right]\\
\leq \;&\frac{ \phi_0-\underline{f}}{K t_2}+\frac{5L^2d_x\sum_{k=0}^{K-1}r_k^2}{4K}+\frac{L^3\alpha^2d_x\sum_{k=0}^{K-1}r_k^2}{N t_2K}\\
&+\frac{12\beta^2L^2\sigma_2^2 p\gamma d_x\sum_{k=0}^{K-1}r_k^2}{t_2K}+\frac{\beta L^2 d_x\sum_{k=0}^{K-1}r_k^2}{8t_2K} + \frac{L^2d_x\sum_{k=0}^{K-1}r_k^2}{8Nt_2K}.
\end{align*}
Substituting $\gamma\leq\frac{1}{\sqrt{KN}}$ and $\sum_{k=0}^{K-1}r_k^2\leq \frac{1}{b}$ into the above inequality, we can get
\begin{align*}
\frac{1}{K}\sum_{k=0}^{K-1}\mathbb{E}\left[\|\mathfrak{g}(x_k,y_k)\|^2\right]\leq \;&   \mathcal{O}\left(\sqrt{\frac{N}{K}}\right) .
\end{align*}
Finally, we have $$\min_{k=0,\cdots,K-1}\mathbb{E}\left[\|\mathfrak{g}(x_k,y_k)\|\right]\leq \frac{1}{K}\sum_{k=0}^{K-1}\mathbb{E}\left[\|\mathfrak{g}(x_k,y_k)\|\right]\leq \mathcal{O}\left(\left(\frac{N}{K}\right)^{\frac{1}{4}}\right),$$
which is the desired result and finishes the proof.

\section{Proof of Lemma \ref{lemma:stationary_to_kkt}} \label{app:proof_of_lemma_stk}
Denote $\delta=\min_{j\in\mathcal{J}}((\overline{y})_j-(y)_j)$ and $\overline{m}=\max_{j\in\mathcal{J}} (\overline{y})_j$. We also denote $\hat{y}=y+\beta c(x)$ and $\mathcal{P}_j[\cdot]=\mathcal{P}_{[0,(\overline{y})_j]}[\cdot]$ for simplicity. By the definitions of the KKT gap and stationarity measure, we have $\left\| \nabla_x f(x,y)\right\|= \|\mathfrak{g}_x(x,y)\|\leq \|\mathfrak{g}(x,y)\|.$ 

For the primal infeasibility and complementarity gap, we consider three cases for $(\hat{y})_j$: (1) $0\leq (\hat{y})_j \leq (\overline{y})_j$; (2) $(\hat{y})_j< 0$; (3) $(\hat{y})_j>(\overline{y})_j$.

In the first case, we have $\mathcal{P}_j[(\hat{y})_j]=(\hat{y})_j$, so
$(\mathfrak{g}_y(x,y))_j=\frac{1}{\beta}((y)_j-(\hat{y})_j)=-c_j(x).$
Hence, we have $[c_j(x)]_+\leq |c_j(x)|=|(\mathfrak{g}_y(x,y))_j|$ and $(y)_j|c_j(x)|\leq \overline{m}|(\mathfrak{g}_y(x,y))_j|$.

In the second case, we have $\mathcal{P}_j[(\hat{y})_j]=0$. Then, we can get
$|(\mathfrak{g}_y(x,y))_j|=\frac{1}{\beta}(y)_j\geq 0$ and $(\hat{y})_j=(y)_j+\beta c_j(x)< 0$, which leads to $c_j(x)<0$. Therefore, it yields $[c_j(x)]_+=0$. Since $|c_j(x)|\leq \Lambda$ under Assumption \ref{assump:lipschitz}, $(y)_j|c_j(x)|\leq\beta \Lambda|(\mathfrak{g}_y(x,y))_j|$.

In the third case, we have $\mathcal{P}_j[(\hat{y})_j]=(\overline{y})_j$ and $|(\mathfrak{g}_y(x,y))_j|=\frac{1}{\beta}((\overline{y})_j-(y)_j)\geq \frac{\delta}{\beta}$. Since $(\hat{y})_j=(y)_j+\beta c_j(x)> (\overline{y})_j$, we can get $c_j(x)>\frac{\delta}{\beta}>0.$
Therefore, in this case $[c_j(x)]_+=c_j(x)>0$. By Assumption \ref{assump:lipschitz}, we have
$[c_j(x)]_+\leq \Lambda \leq \frac{\Lambda \beta}{\delta}|(\mathfrak{g}_y(x,y))_j|,$
and similarly
$(y)_j|c_j(x)|\leq (\overline{y})_j\Lambda\leq \frac{\overline{m}\Lambda \beta}{\delta}|(\mathfrak{g}_y(x,y))_j|,$
where we used $|(\mathfrak{g}_y(x,y))_j|\geq \frac{\delta}{\beta}$.

Denote $\hat{c}=\max(1, \frac{\Lambda\beta}{\delta})$. Then, we have $[c_j(x)]_+\leq \hat{c}|(\mathfrak{g}_y(x,y))_j|\leq \hat{c}\|\mathfrak{g}_y(x,y)\|$ and $(y)_j|c_j(x)|\leq \overline{m}\hat{c}|(\mathfrak{g}_y(x,y))_j|\leq \overline{m}\hat{c} \|\mathfrak{g}_y(x,y)\|$
Let $C=1+(\overline{m}+1)\hat{c}$. Finally, we can conclude
$\mathcal{K}(x,y)\leq C \|\mathfrak{g}(x,y)\|$.

\section{Detailed Experimental Settings}\label{app:experiemnt_setting}
\subsection{Energy Management}
We consider a classic energy management problem in power systems called load curtailment. In this problem, a load aggregator tries to coordinate the loads of multiple users within a distribution network to meet the load requirements imposed by the higher-level grid operator. On the one hand, the aggregator needs to ensure that the total power injection into the network satisfies a constraint tied to a reference load. On the other hand, the operational costs associated with users' load adjustments should be minimized to maintain a satisfactory consumer experience.


Mathematically, denote $x\in\mathbb{R}^{d_x}$ as the power load of multiple users. Let $D$ denote the reference load received from the grid operator, which imposes a constraint on the distribution network’s net power exchange with the main grid. The power injection of the distribution network is not simply the sum of users' loads but is determined by nonlinear power flow dynamics. Denote the dynamics as a function of the load levels of multiple users $p_c(x):\mathbb{R}^{d_x}\to \mathbb{R}$. In our formulation, $p_c(x)$ is viewed as a black box, provided that the topology and parameters of the distribution network are unknown. That means we can only observe the total power consumption \( p_c(x) \) of a distribution network given the power load $x$ of users.

\begin{figure}
    \centering
    \includegraphics[width=0.7\linewidth]{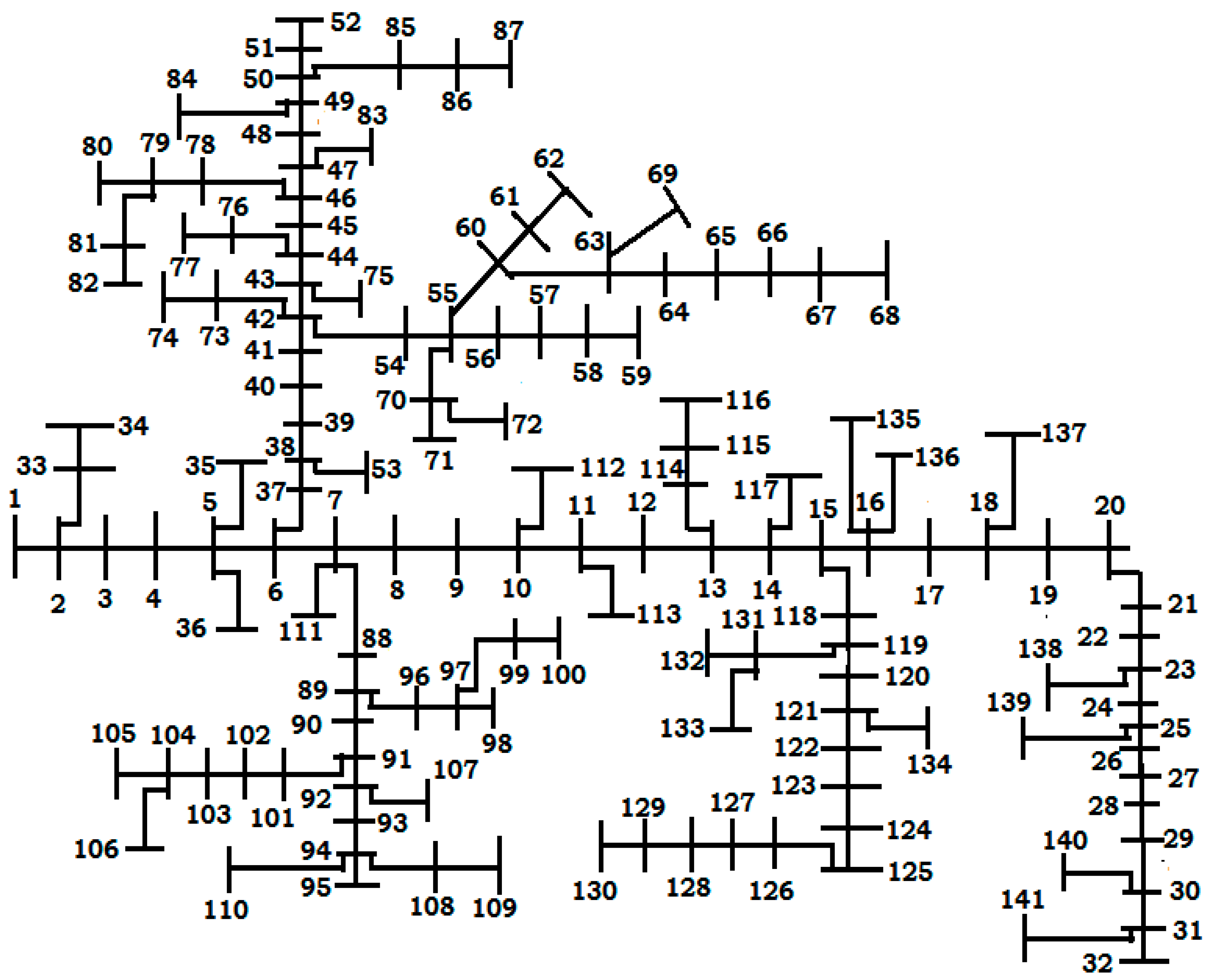}
    \caption{141-bus distribution network}
    \label{fig:141_network}
\end{figure}

We apply the 141-bus distribution network model as the nonconvex black box (shown in Figure \ref{fig:141_network}) \citep{khodr2008maximum}. 
We consider the following problem with $d_x=168$:
\begin{equation}
\begin{aligned}\label{eq:energy_problem}
&\min_{x\in\mathbb{R}^{168}} h(x)=\sum_{i\in[168]}c_i(x(i))+\rho(x),\quad \text{s.t.}\; p_c(x)\leq D,\;x\in \mathcal{X}= [\underline{x},\overline{x}],
\end{aligned}    
\end{equation}
where $c_i:\mathbb{R}\to \mathbb{R}$ is the cost function of user $i$ and defined as $c_i(x(i))=a_i x^2(i)+b_ix(i)$. $\rho (x):\mathcal{X}\to \mathbb{R}$ is the penalty term when the voltage is out of the standard region and formulated as 
$$\rho(x)=\sum_{j\in [141]} \left( \max(v_j(x)-\overline{v},0)^2+\max(\underline{v}-v_j(x),0)^2  \right).$$
Here $v_j(x)$ denotes the voltage of the $j$th node when the power load of the distribution network is $x$, and $\underline{v},\overline{v}$ represent the lower and upper bounds of the voltage. $v_j:\mathbb{R}^{168}\to\mathbb{R},\forall j$ is also a black-box mapping with only the function value observable. Therefore, the objective and constraint functions in problem (\ref{eq:energy_problem}) are both non-analytical.

\textbf{Parameters Setting.}
In our numerical simulation, $\underline{x}$ is also set to 0, and $\overline{x}$ is the nominal load level from the original 141-bus system. The coefficients of cost functions, $a_i$ and $b_i$, are randomly sampled from the intervals $(0.5,1.5)$ and $(0,5)$, respectively. For the voltage penalty, we set $\underline{v}=0.96 \; p.u.$ and $\overline{v}=1.04\; p.u.$.  The total load to be curtailed is set to $0.15 \;p.u.=1500\; kW$. 
In addition to testing our proposed algorithms, we also compare them with three other algorithms, ZO-MinMax, ZOAGP, and stochastic zeroth-order constraint extrapolation (SZO-ConEx). For all tests, we compute $\|\mathfrak{g}(x_k,y_k)\|$ as the stationarity measure. The constraint violation is measured by $p_c(x)-D$. In Table 2, the relative error is computed by $(h(x_k)-h^*)/h^*$ based on the optimal objective function value $h^*=0.3841$.

For ZOB-GDA, we consider four scenarios with batch sizes $b=1,10,50,168$, where the step sizes are set as $\alpha=0.03, 0.015, 0.005, 0.00035$ and $\beta=0.01\alpha$, respectively. For ZOB-SGDA, we set $p=10$ and $\gamma=0.3$, and adopt the same batch-size scenarios and corresponding step sizes are set as $\alpha=0.03, 0.025, 0.005, 0.00035$ and $\beta=0.01\alpha$. For the benchmark algorithms, ZO-MinMax is implemented with $\alpha=\beta=5\times 10^{-6}$; since it involves constraint-handling techniques, we adopt a decaying penalty parameter $\delta_k = \min(50/k,0.1)$. For SZO-ConEx, we set $\alpha=\beta=5\times 10^{-6}$. For ZOAGP, we choose $\alpha=\beta=0.00035$. We set the maximum iteration steps as $K=20000$.
For the smoothing radius, ZOB-GDA, ZOB-SGDA, and ZOAGP use $r_k = \min(10^{-1}/k^{1.2}, 2\times 10^{-4})$, while SZO-ConEx and ZO-MinMax use $r_k = \min(10^{-2}/(k+4000)^{1.1}, 1\times 10^{-5})$.
All parameters are selected for the best performance among multiple options. All experiments are conducted on a MacBook Pro laptop equipped with an Apple M1 Pro SoC (10-core CPU: 8 performance cores and 2 efficiency cores) and 16 GB of unified memory.

\subsection{High-Dimensional Parameter Optimization}
In the 1000-dimensional constrained nonconvex optimization problem, each algorithm is tested 20 times.

\textbf{Parameter Settings} The matrix $B\in\mathbb{R}^{1000\times1000}$ and the constraint vector $q\in\mathbb{R}^{1000}$ are generated randomly with i.i.d. entries following $\mathcal{N}(0,\sigma^2)$ with $\sigma=1/10\sqrt{10}$, and remain fixed throughout each run. We set $\Bar{c}=0.5$ for all tests. For ZOB-GDA, we apply a tuned block size $b=30$, and the step sizes are set as $\alpha=0.5$ and $\beta=0.5$. For ZOB-SGDA, we set $p=0.1$, $\gamma=0.9$, $\alpha=0.5$ and $\beta=5$. For the comparison algorithms, ZO-MinMax is implemented with $\alpha=\beta=0.2$. For SZO-ConEx, we set $\alpha=0.3$ and $\beta=0.003$. For ZOAGP, we set $\alpha=0.1$ and $\beta=5$. The maximum budget of function evaluations is set as $1\times 10^6$.
For the smoothing radius, ZOB-GDA and ZOB-SGDA use $r_k = \min(10^{-2}/k, 1\times 10^{-3})$; ZO-MinMax and SZO-ConEx use $r_k = \min(10^{-2}/k, 2\times 10^{-4})$; ZOAGP uses $r_k = \min(10^{-3}/k, 1\times 10^{-4})$.
All parameters are selected for the best performance among multiple options.

\end{document}